\documentclass[12pt]{amsart}

\usepackage{amssymb}
\usepackage{amsmath}
\usepackage{amsthm}
\usepackage{graphics}

\allowdisplaybreaks

\numberwithin{equation}{section}

\theoremstyle{plain}
\newtheorem{thm}{Theorem}[section]
\newtheorem{prop}[thm]{Proposition}
\newtheorem{lem}[thm]{Lemma}
\newtheorem{cor}[thm]{Corollary}

\theoremstyle{definition}

\newcommand{\ichi}{\mathbf{1}}
\newcommand{\R}{\mathbb{R}}

\newcommand{\Z}{\mathbb{Z}}
\newcommand{\N}{\mathbb{N}}

\newcommand{\calF}{\mathcal{F}}

\newcommand{\calS}{\mathcal{S}}
\newcommand{\supp}{\mathrm{supp}\, }
\newcommand{\suppast}{\mathrm{supp}{^{\ast}}\, }

\begin{document}

\title{Bilinear pseudo-differential operators with exotic symbols}

\author{Akihiko Miyachi \and Naohito Tomita}
\date{}

\address{Akihiko Miyachi \\
Department of Mathematics \\
Tokyo Woman's Christian University \\
Zempukuji, Suginami-ku, Tokyo 167-8585, Japan}
\email{miyachi@lab.twcu.ac.jp}

\address{Naohito Tomita \\
Department of Mathematics \\
Graduate School of Science \\
Osaka University \\
Toyonaka, Osaka 560-0043, Japan}
\email{tomita@math.sci.osaka-u.ac.jp}



\begin{abstract}
The boundedness 
from $L^p \times L^q$ to $L^r$, $1<p,q \le \infty$,
$0<1/p+1/q=1/r \le 1$, of 
bilinear pseudo-differential operators
with symbols in the bilinear H\"ormander class
$BS^m_{\rho,\rho}$, $0 \le \rho <1$, is proved for 
the critical order $m$. 
Related results for the cases $p=1$, $q=1$ or $r=\infty$
are also obtained.
\end{abstract}

\maketitle

\section{Introduction}\label{section1}
Let $m \in \R$ and $0 \le \delta \le \rho \le 1$.
We say that a function
$\sigma(x,\xi,\eta) \in C^{\infty}(\R^n \times \R^n \times \R^n)$
belongs to the bilinear H\"ormander symbol class
$BS^m_{\rho,\delta}=BS^m_{\rho,\delta}(\R^n)$
if for every triple of multi-indices
$\alpha,\beta,\gamma \in \N_0^n=\{0,1,2,\dots\}^n$
there exists a constant $C_{\alpha,\beta,\gamma}>0$ such that
\[
|\partial^{\alpha}_x\partial^{\beta}_{\xi}
\partial^{\gamma}_{\eta}\sigma(x,\xi,\eta)|
\le C_{\alpha,\beta,\gamma}
(1+|\xi|+|\eta|)^{m+\delta|\alpha|-\rho(|\beta|+|\gamma|)}.
\]
For a symbol $\sigma \in BS^{m}_{\rho,\delta}$,
the bilinear pseudo-differential operator
$T_{\sigma}$ is defined by
\[
T_{\sigma}(f,g)(x)
=\frac{1}{(2\pi)^{2n}}
\int_{\R^n \times \R^n}e^{i x \cdot(\xi+\eta)}
\sigma(x,\xi,\eta)\widehat{f}(\xi)\widehat{g}(\eta)\, d\xi d\eta,
\qquad f,g \in \calS(\R^n).
\]

The study of bilinear operators $T_{\sigma}$ with $\sigma$ in 
the bilinear H\"ormander class $BS^m_{\rho,\delta}$ 
was initiated by B\'enyi, Maldonado, Naibo, and Torres in \cite{BMNT}, 
where in particular the symbolic calculus of the operators $T_{\sigma}$, 
$\sigma \in BS^m_{\rho,\delta}$, was established.  
The boundedness properties of those operators have been considered in many works, 
some of which will be mentioned below. 
In the present paper, we shall also consider the boundedness property of 
the operators $T_{\sigma}$, $\sigma \in BS^{m}_{\rho,\delta}$. 
For the boundedness of the operators $T_{\sigma}$, we shall 
use the following terminology. 
If $X,Y,Z$ are function spaces on $\R^n$ 
equipped with quasi-norms 
$\|\cdot \|_{X},\,\|\cdot \|_{Y},\,\|\cdot \|_{Z}$ 
and if there exists a constant $A_{\sigma}$ such that 
the estimate 
\begin{equation}\label{boundedness-XYZ}
\|T_{\sigma}(f,g)\|_{Z}
\le A_{\sigma} \|f\|_{X} \|g\|_{Y}, 
\quad 
f\in \calS \cap X, 
\quad 
g\in \calS \cap Y,  
\end{equation}
holds, then we shall simply say that 
$T_{\sigma}$ is bounded from $X\times Y$ to $Z$ and write 
\[
T_{\sigma}: X\times Y \to Z.
\] 
The smallest constant $A_{\sigma}$ of \eqref{boundedness-XYZ} 
is denoted by 
$\|T_{\sigma}\|_{X\times Y \to Z}$.

In the case $\rho=1$,
bilinear pseudo-differential operators
with symbols in $BS^0_{1,\delta}$, $\delta<1$,
fall into the bilinear Calder\'on-Zygmund theory
in the sense of Grafakos-Torres \cite{GT} and 
their boundedness properties are well-understood; 
see, e.g., Coifman-Meyer \cite{CM}, 
B\'enyi-Torres \cite{BT}, 
and B\'enyi-Maldonado-Naibo-Torres \cite{BMNT}.
In the case $\rho<1$, however, 
we cannot reduce the corresponding operators
to bilinear Calder\'on-Zygmund operators 
and there are some interesting features peculiar to 
the bilinear case. 
For example, 
in contrast to the well-known 
Calder\'on-Vaillancourt theorem (\cite{CV})
for linear pseudo-differential operators, 
the condition $\sigma \in BS^0_{\rho,\rho}$, $0\le \rho <1$, 
does not assure any boundedness of the corresponding 
bilinear operator. 
This gap between the linear and bilinear cases
was first pointed out by B\'enyi-Torres \cite{BT-2} 
for the case $\rho=0$.

The subject of the present paper concerns with the estimate 
\begin{equation}\label{boundedness-pqr}
T_{\sigma}: H^p \times H^q \to L^r,
\quad 
\frac{1}{p}+ \frac{1}{q}=\frac{1}{r},  
\quad 
\sigma \in BS^{m}_{\rho,\rho}, 
\quad 
0\le \rho <1, 
\end{equation}
where $H^p$ denotes Hardy space and 
$L^r$ denotes Lebesgue space. 
In the case $p=q=r=\infty$, 
instead of $L^{\infty}\times L^{\infty} \to L^{\infty}$, 
we shall consider $L^{\infty}\times L^{\infty} \to BMO$.

For $0 \le \rho <1$
and for $0<p,q,r \le \infty$ satisfying $1/p+1/q=1/r$,
we define 
\begin{align*}
&
m_\rho(p,q)
=(1-\rho)m_0(p,q),
\\
&
m_0(p,q)
=-n \left(\max\left\{
\frac{1}{2}, \,
\frac{1}{p}, \,
\frac{1}{q}, \,
1-\frac{1}{r}, \, 
\frac{1}{r} -\frac{1}{2}
\right\}\right). 
\end{align*}
Here is an expression of $m_0(p,q)$ that will be easy to see. 
We divide the region of $(1/p,1/q)$ into 5 regions
$J_0,\dots,J_4$ as follows:
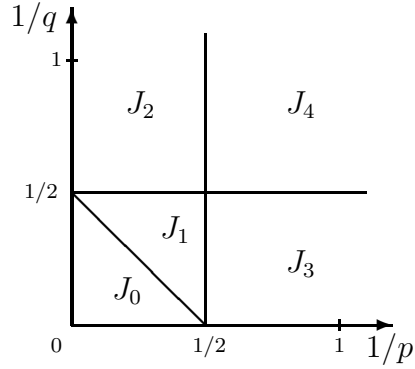
\begin{center}
\begin{picture}(160,150)
\thicklines
\put(20,20){\vector(1,0){120}}
\put(20,20){\vector(0,1){120}}
\put(18,120){\line(1,0){4}}
\put(120,18){\line(0,1){4}}
\put(20,70){\line(1,-1){50}}
\put(20,70){\line(1,0){110}}
\put(70,20){\line(0,1){110}}
\put(130,8){$1/p$}
\put(-3,133){$1/q$}
\put(35,30){$J_0$}
\put(53,53){$J_1$}
\put(40,100){$J_2$}
\put(100,40){$J_3$}
\put(100,100){$J_4$}
\thinlines
\put(12,10){{\tiny $0$}}
\put(65,10){{\tiny $1/2$}}
\put(118,10){{\tiny $1$}}
\put(2,68){{\tiny $1/2$}}
\put(12,118){{\tiny $1$}}
\end{picture}
\end{center}
Then 
\[
m_0(p,q)
=
\begin{cases}
\frac{n}{r}-n
&\quad \text{if} \quad
\left(\frac{1}{p},\frac{1}{q}\right) \in J_0; 
\\
-\frac{n}{2}
&\quad \text{if} \quad
\left(\frac{1}{p},\frac{1}{q}\right) \in J_1; 
\\
-\frac{n}{q}
&\quad \text{if} \quad
\left(\frac{1}{p},\frac{1}{q}\right) \in J_2;
\\
-\frac{n}{p}
&\quad \text{if} \quad
\left(\frac{1}{p},\frac{1}{q}\right) \in J_3;
\\
\frac{n}{2}-\frac{n}{r}
&\quad \text{if} \quad
\left(\frac{1}{p},\frac{1}{q}\right) \in J_4,
\end{cases}
\]
where $1/p+1/q=1/r$.

The number $m_{\rho}(p,q)$ 
is the critical order as 
the following proposition shows. 
A proof of this proposition will be given in Appendix of this paper. 
\begin{prop}\label{critical-order}
Let $0 \le \rho <1$, $0<p,q,r \le \infty$, and suppose 
$1/p+1/q=1/r$. 
If $r<\infty$, then 
\[
m_{\rho}(p,q)
=
\sup \{m \in \R \,:\, 
T_{\sigma}: H^p \times H^q \to L^r 
\;\; \text{for all}\;\; \sigma \in BS^{m}_{\rho,\rho}
\}.  
\] 
When $p=q=r=\infty$, the above equality holds 
if we replace 
$H^p \times H^q \to L^r$ by 
$L^{\infty} \times L^{\infty} \to BMO $. 
\end{prop}

It should be an interesting problem 
to prove the boundedness of
bilinear pseudo-differential operators
in the critical class $BS^m_{\rho,\rho}$,
$m=m_{\rho}(p,q)$. 
For the case $\rho=0$, this problem was solved by 
the authors in \cite{Miyachi-Tomita}. 
For the case $0<\rho <1$, 
to the best of the authors' knowledge, 
the only known result for the problem is due to 
Naibo \cite{Naibo}, which however is restricted to 
the case $0<\rho <1/2$ and $p=q=r=\infty$. 
The purpose of the present paper is to solve the problem 
in the range $0 \le 1/p+1/q=1/r \le 1$.

The following are the main results of this paper. 

\begin{thm}\label{main-thm-1}
Let $0 \le \rho <1$
and $m=-(1-\rho)n/2$.
Then all bilinear pseudo-differential operators
with symbols in $BS^{m}_{\rho,\rho}(\R^n)$
are bounded from $L^2(\R^n) \times L^{\infty}(\R^n)$
to $L^2(\R^n)$.
\end{thm}

\begin{thm}\label{main-thm-2}
Let $0 \le \rho <1$
and $m=-(1-\rho)n$.
Then all bilinear pseudo-differential operators
with symbols in $BS^{m}_{\rho,\rho}(\R^n)$
are bounded from $L^{\infty}(\R^n) \times L^{\infty}(\R^n)$ 
to $BMO (\R^n)$.
\end{thm}

\begin{cor}\label{main-cor}
Let $0 \le \rho <1$, $1 \le p,q,r \le \infty$, 
$1/p+1/q=1/r$, and $m=m_{\rho}(p,q)$. 
Then all bilinear pseudo-differential operators
with symbols in $BS^{m}_{\rho,\rho}(\R^n)$
are bounded from $L^p(\R^n) \times L^{q}(\R^n)$ to $L^r(\R^n)$, 
where $L^p(\R^n)$ (respectively, $L^q(\R^n)$)
should be replaced by $H^p(\R^n)$ (respectively, $H^q(\R^n)$)
if $p=1$ (respectively, $q=1$) 
and $L^r(\R^n)$ should be replaced by $BMO(\R^n)$ if $r=\infty$.
\end{cor}

Here are some comments on the previous works related 
to the above results. 
For the subcritical case $m<m_{\rho}(p,q)$, 
the boundedness \eqref{boundedness-pqr} were 
obtained by 
Michalowski-Rule-Staubach \cite{MRS} (for $(1/p, 1/q)$ in the triangle 
with vertices $(1/2, 1/2)$, $(1/2, 0)$, $(0, 1/2)$) 
and by 
B\'enyi-Bernicot-Maldonado-Naibo-Torres \cite{BBMNT} 
(in the range $1/p+1/q \le 1$). 
As we mentioned above, the 
case $m=m_{\rho}(p,q)$ with $\rho=0$ 
was obtained by the authors \cite{Miyachi-Tomita}. 
In fact, \cite[Theorem 1.1]{Miyachi-Tomita} 
gives a sharper version of the above Corollary \ref{main-cor} for $\rho=0$ 
and covers the full range $0<p,q,r \le \infty$. 
Naibo \cite{Naibo} has proved the claim of Theorem \ref{main-thm-2} 
in the case $0<\rho<1/2$.

Theorem \ref{main-thm-1} should be one of the key estimates
to consider the critical case $m=m_{\rho}(p,q)$ in the 
whole range $0<p,q \le \infty$. 
Here is a comment concerning the method of proof of Theorem \ref{main-thm-1}.  
As we mentioned above this theorem for the case 
$\rho=0$ was already proved in \cite{Miyachi-Tomita}. 
However, the method of the present paper is 
totally different from that of \cite{Miyachi-Tomita}. 
The method of \cite{Miyachi-Tomita} seems to work only 
in the case $\rho=0$, but the method of the present paper 
covers all $0\le \rho <1$.

The contents of this paper are as follows.
In Section \ref{section2},
we recall some preliminary facts.
In Sections \ref{section3}, \ref{section4} and \ref{section5},
we prove Theorems \ref{main-thm-1}, \ref{main-thm-2}
and Corollary \ref{main-cor}, respectively.
In Appendix \ref{appendix}, we prove Proposition \ref{critical-order}.

\section{Preliminaries}\label{section2}

For two nonnegative quantities $A$ and $B$,
the notation $A \lesssim B$ means that
$A \le CB$ for some unspecified constant $C>0$,
and $A \approx B$ means that
$A \lesssim B$ and $B \lesssim A$.
We denote by $\ichi_S$ the characteristic function of a set $S$,
and by $|S|$ the Lebesgue measure of a measurable set $S$ in $\R^n$.

Let $\calS(\R^n)$ and $\calS'(\R^n)$ be the Schwartz spaces of
all rapidly decreasing smooth functions
and tempered distributions,
respectively.
We define the Fourier transform $\calF f$
and the inverse Fourier transform $\calF^{-1}f$
of $f \in \calS(\R^n)$ by
\[
\calF f(\xi)
=\widehat{f}(\xi)
=\int_{\R^n}e^{-ix\cdot\xi} f(x)\, dx
\quad \text{and} \quad
\calF^{-1}f(x)
=\frac{1}{(2\pi)^n}
\int_{\R^n}e^{ix\cdot \xi} f(\xi)\, d\xi.
\]
For $m \in L^{\infty}(\R^n)$,
the linear Fourier multiplier operator $m(D)$
is defined by
\[
m(D)f(x)
=\calF^{-1}[m\widehat{f}](x)
=\frac{1}{(2\pi)^n}\int_{\R^n}
e^{ix\cdot\xi}m(\xi)\widehat{f}(\xi)\, d\xi,
\quad
f \in \calS(\R^n).
\]

We recall the definition 
of Hardy spaces and the space $BMO$ on $\R^n$
(see \cite[Chapters 3 and 4]{Stein}).
Let $0 < p \le \infty$, and let $\phi \in \calS(\R^n)$ be such that
$\int_{\R^n}\phi(x)\, dx \neq 0$. 
Then the Hardy space $H^p(\R^n)$ consists of
all $f \in \calS'(\R^n)$ such that
\[
\|f\|_{H^p}=\left\|\sup_{0<t<\infty}|\phi_t*f|\right\|_{L^p}<\infty,
\]
where $\phi_t(x)=t^{-n}\phi(x/t)$.
It is known that $H^p(\R^n)$ does not depend 
on the choice of the function $\phi$
and $H^p(\R^n)=L^p(\R^n)$ for $1<p \le \infty$.
The space $BMO(\R^n)$ consists of
all locally integrable functions $f$ on $\R^n$ such that
\[
\|f\|_{BMO}
=\sup_{Q}\frac{1}{|Q|}
\int_{Q}|f(x)-f_Q|\, dx<\infty, 
\]
where $f_Q$ is the average of $f$ on $Q$
and the supremum is taken over all cubes $Q$ in $\R^n$.
It is known that the dual spaces of
$H^1(\R^n)$ is $BMO(\R^n)$.

We end this section by quoting the following, 
which we shall call {\it Schur's lemma}\/. 
For a proof, see, e.g., \cite[Appendix I]{Grafakos-1}.

\begin{lem}[Schur's lemma]
Let $\{A_{j,k}\}_{j,k \ge 0}$ be a sequence of nonnegative numbers 
satisfying
\[
\sup_{j \ge 0}\sum_{k \ge 0}A_{j,k}\le 1 
\quad \text{and}\quad 
\sup_{k \ge 0}\sum_{j \ge 0}A_{j,k}
\le 1.
\]
Then 
\[
\sum_{j,k \ge 0}A_{j,k}b_j c_k
\le \left(\sum_{j \ge 0}b_j^2\right)^{1/2}
\left(\sum_{k \ge 0}c_k^2\right)^{1/2}
\]
for all nonnegative sequences $\{b_j\}$ and $\{c_k\}$.  
\end{lem}

\section{Proof of Theorem \ref{main-thm-1}}\label{section3}

In this section,
we shall prove Theorem \ref{main-thm-1}.
The argument is divided into three subsections. 
Although the proof for general $\sigma \in BS^{m}_{\rho,\rho}$ 
is somewhat complicated, 
the main idea already consists in the special 
case that $\sigma (x,\xi,\eta)$ is independent of $x$, 
namely the bilinear Fourier multiplier case. 
In this case, 
$\sigma_{j,k,\nu}$ to be introduced 
in Subsection \ref{subsection-decomposition} reduces to 
\[
\sigma_{j,k,\nu}
=
\left\{
\begin{array}{ll}
{\sigma_{j,\nu}} & {\qquad\text{if}\quad k=0}
\\ 
{0} & {\qquad\text{if}\quad k\ge 1,}
\end{array}
\right.
\]
and the argument will be simple.

We use the following notation and terminology. 
For a finite set $\Lambda$, we write $|\Lambda|$ to denote 
the number of elements of $\Lambda$. 
The following are cubes in $\R^n$: 
\begin{align*}
&Q=[-1,1]^{n}, \quad 
aQ=[-a,a]^{n}, \quad a>0, 
\\
&
x+aQ=\{x+ y \,:\, y\in aQ\}, 
\quad x\in \R^n.
\end{align*} 
If $\sigma$ is 
a function on $\R^n \times \R^n \times \R^n$, 
then 
\[
\suppast \sigma=\;\text{closure of}\;
\{(\xi, \eta)\in \R^n \times \R^n 
\,:\, 
\sigma (x,\xi,\eta) \neq 0 \;\;\text{for some}\;\; x\in \R^n\}. 
\]
The usual inner product of $f,h \in L^2=L^2 (\R^n)$ is denoted by 
$\langle f, g\rangle$. 
If $\{E_{\alpha}\}$ is a finite family of subsets of $\R^n$, 
$L$ is a positive integer, and if 
\[
|\{\beta \,:\, E_{\beta} \cap E_{\alpha} 
\neq \emptyset \}|\le L
\quad \text{for all}\quad \alpha, 
\]
then 
we say that the {\it interaction} 
of the family $\{E_{\alpha}\}$ is bounded by $L$.

\subsection{Decomposition of the symbol and some preliminaries}
\label{subsection-decomposition}

We use the following two types of partitions of unity.
One is the dyadic decomposition:
\begin{equation}\label{dyadic-decomposition}
\begin{split}
&\mathrm{supp}\, \psi_0
\subset \{\zeta \in \R^d \,:\, |\zeta| \le 2\},
\\
&\mathrm{supp}\, \psi_j
\subset \{\zeta \in \R^d \,:\, 2^{j-1} \le |\zeta| \le 2^{j+1}\},
\quad j \ge 1,
\\
&\|\partial^{\alpha}\psi_j\|_{L^{\infty}} \lesssim 2^{-j|\alpha|},
\quad \alpha \in \N_0^d, \ j \ge 0,
\\
&\sum_{j \ge 0}\psi_j(\zeta)=1,
\quad \zeta \in \R^d.
\end{split}
\end{equation}
The other is the uniform decomposition:
\begin{equation}\label{uniform-decomposition}
\begin{split}
&\mathrm{supp}\, \varphi \subset Q,
\\
&\sum_{\nu \in \Z^n}\varphi(\xi-\nu)=1,
\quad \xi \in \R^n.
\end{split}
\end{equation}
Here $\psi_j$, $j \ge 0$, and $\varphi$
are smooth real-valued functions. 
We shall use \eqref{dyadic-decomposition} 
with $d=2n$ and $d=n$. 
We write $\Psi_j$ 
to denote the function $\psi_j$ of 
\eqref{dyadic-decomposition} with $d=2n$ 
and write $\psi_j$ to denote the function 
of \eqref{dyadic-decomposition} 
with $d=n$. 
We shall use \eqref{uniform-decomposition} only on $\R^n$.

In this subsection, we assume 
$\sigma \in BS^{m}_{\rho,\rho}$ with $m\in \R$ and $0\le \rho \le 1$.  
(The conditions on $m$ and $\rho$ 
as in Theorem \ref{main-thm-1} 
are not necessary in this subsection.)
We decompose $\sigma$ as
\begin{equation}\label{symbol-dec-0}
\sigma(x,\xi,\eta)=
\sum_{j \ge 0}
\sum_{\nu=(\nu_1,\nu_2) \in \Z^n \times \Z^n}
\sigma_{j,\nu}(x,\xi,\eta)
=
\sum_{j \ge 0}\sum_{k \ge 0}
\sum_{\nu=(\nu_1,\nu_2) \in \Z^n \times \Z^n}
\sigma_{j,k,\nu}(x,\xi,\eta), 
\end{equation}
where 
\begin{equation}\label{def-sigmajn}
\sigma_{j,\nu}(x,\xi,\eta)
=\sigma(x,\xi,\eta)
\varphi(2^{-j\rho}\xi-\nu_1)\varphi(2^{-j\rho}\eta-\nu_2)
\Psi_j(\xi,\eta)
\end{equation}
and 
\begin{equation}\label{def-sigmajkn}
\begin{split}
\sigma_{j,k,\nu}(x,\xi,\eta)
&=[\psi_k(2^{-j\rho}D_x)\sigma_{j,\nu}](x,\xi,\eta)
\\
&=
2^{j\rho n}
\int_{\R^n}[\calF^{-1}\psi_k](2^{j\rho}y)
\sigma_{j,\nu}(x-y,\xi,\eta)\, dy. 
\end{split}
\end{equation}

Notice the following facts.  
First, 
if we write the projections as 
\[
\pi_1 (\xi, \eta)=\xi, 
\quad 
\pi_2 (\xi, \eta)=\eta, 
\] 
then it is obvious that 
\begin{equation}\label{replace-f1-g1}
T_{\sigma_{j,k,\nu}} (f,g)
=
T_{\sigma_{j,k,\nu}} (f^{(1)},g)
=T_{\sigma_{j,k,\nu}} (f,g^{(1)}) 
\end{equation}
whenever $f^{(1)}$ and $g^{(1)}$ satisfy 
$(f^{(1)})^{\wedge}= \widehat{f}$ on 
$\pi_1 (\suppast (\sigma_{j,\nu}))$ 
and $(g^{(1)})^{\wedge}= \widehat{g}$ on 
$\pi_2 (\suppast (\sigma_{j,\nu}))$.  
Secondly, the Fourier transform of 
$T_{\sigma_{j,k,\nu}}(f,g)$ is given by
\begin{equation}\label{Fourier-Tsjkn}
\begin{split}
&\calF [T_{\sigma_{j,k,\nu}}(f,g)] (\zeta) 
\\
&=
\frac{1}{(2\pi)^{2n}}
\int_{\R^n \times \R^n}
\psi_{k}(2^{-j\rho}(\zeta-\xi-\eta))
[\calF_x\sigma_{j,\nu}](\zeta-\xi-\eta,\xi,\eta)
\widehat{f}(\xi)\widehat{g}(\eta)\, d\xi d\eta,
\end{split}
\end{equation}
where $\zeta \in \R^n$ 
and $\calF_x\sigma_{j,\nu}$ denotes 
the partial Fourier transform
of $\sigma_{j,\nu}(x,\xi,\eta)$ with 
respect to the $x$-variable. 
From this we see that 
\begin{equation}\label{FouriersuppTsjkn}
\supp \calF [T_{\sigma_{j,k,\nu}}(f,g)] 
\subset 
\bigcup_{(\xi, \eta)\in \suppast (\sigma_{j,\nu})}
\{ 
\zeta 
\,:\, 
|\zeta - \xi - \eta|
\le 2^{j \rho + k+1}
\}. 
\end{equation}
Hence, we have  
\begin{equation}\label{replace-h1}
\big\langle T_{\sigma_{j,k,\nu}} (f,g), h \big\rangle 
=
\big\langle T_{\sigma_{j,k,\nu}} (f,g), h^{(1)} \big\rangle 
\end{equation}
whenever $h^{(1)}$ satisfy $(h^{(1)})^{\wedge}= \widehat{h}$ 
on the set on the right-hand side of \eqref{FouriersuppTsjkn}. 
In the argument to follow, we shall use 
\eqref{replace-f1-g1} and \eqref{replace-h1} 
by choosing the functions $f^{(1)}$, $g^{(1)}$, $h^{(1)}$ according to  
several different situations.

We also use the following general lemma 
for nearly orthogonal functions and operators. 
\begin{lem}\label{orthogonality}
{\rm (1)} If $\{f_{\alpha}\}$ is a finite family of functions in $L^2$, 
$L$ is a positive integer, 
and if $|\{\beta \,:\, \langle f_{\beta}, f_{\alpha}\rangle \neq 0\}|
\le L$ 
for all $\alpha$, then 
$\|\sum_{\alpha} f_{\alpha}\|_{L^2}^2 
\le L \sum_{\alpha} \|f_{\alpha}\|_{L^2}^2$. 
\\
{\rm (2)} If $\{T_{\alpha}\}$ is a finite family of 
bounded linear operators in $L^2$, 
$L$ is a positive integer, 
and if $| \{\beta \,:\, T_{\beta}^{\ast} T_{\alpha} \neq 0\} |
\le L$ 
for all $\alpha$, then 
$\|\sum_{\alpha} T_{\alpha}\|_{L^2\to L^2}^2 
\le L \sum_{\alpha} \|T_{\alpha}\|_{L^2\to L^2}^2$. 
\\
{\rm (3)} If $\{T_{\alpha}\}$ is a finite family of 
bounded linear operators in $L^2$, 
$L$ is a positive integer, 
and if $|\{\beta \,:\, T_{\beta} T_{\alpha}^{\ast} \neq 0\} |
\le L$ 
for all $\alpha$, then 
$\|\sum_{\alpha} T_{\alpha}\|_{L^2\to L^2}^2 
\le L \sum_{\alpha} \|T_{\alpha}\|_{L^2\to L^2}^2$. 
\end{lem}
\begin{proof} 
To prove (1), we write 
\begin{equation*}
\bigg\|\sum_{\alpha} f_{\alpha}\bigg\|_{L^2}^2 
=
\sum_{\alpha}\sum_{\beta}
\langle f_{\alpha}, f_{\beta} \rangle 
\le 
\sum_{\alpha}\sum_{\beta}
\ichi \{
\langle f_{\alpha}, f_{\beta} \rangle 
\neq 0 \} 
\|f_{\alpha}\|_{L^2}
\|f_{\beta}\|_{L^2}. 
\end{equation*}
Applying Schur's lemma, we obtain the desired inequality. 
We can prove (2) by applying (1) to $f_{\alpha}= T_{\alpha}f$. 
The assertion (3) follows from (2) 
since the norms of an operator and its adjoint are the same. 
\end{proof}

\subsection{Basic estimates}
\label{subsection-basic-estimates}
In this subsection, except in the last lemma, Lemma 
\ref{sum-on-annulus}, we only assume 
$\sigma \in BS^{m}_{\rho,\rho}$ with $m\in \R$ and $0\le \rho \le 1$.

We shall give some basic estimates which will be used later. 
We use the following notation 
\[
S_{a}(f) (x) 
=a^{n} \int_{\R^n} 
\frac{|f(y)|}{(1+ a |x-y|)^{n+1}}\, dy, 
\quad a>0, \;\; x\in \R^{n}. 
\]

Let us start with the estimate for the square function of
$\widetilde{\varphi}(2^{-j\rho}D-\ell)f$
with respect to $\ell \in \Z^n$. 
Although this is known to many people,
we shall give the proof for the reader's convenience.
\begin{lem}\label{basic-estimate-0}
Let $\widetilde{\varphi} \in \calS(\R^n)$.
Then 
\[
\bigg(\sum_{\ell \in \Z^n}
|\widetilde{\varphi}(2^{-j\rho}D-\ell)f(x)|^2\bigg)^{1/2}
\lesssim
S_{2^{j\rho}}(f^2)(x)^{1/2}
\]
holds for $j \ge 0$ and $x\in \R^n$. 
\end{lem}
\begin{proof}
Since $\widetilde{\varphi}(2^{-j\rho}D-\ell)f(x)
=\widetilde{\varphi}(D-\ell)[f(2^{-j\rho}\, \cdot \,)](2^{j\rho}x)$,
by a scaling argument,
it is sufficient to prove the case $j=0$.
By a periodization technique,
we can write
\begin{align*}
\widetilde{\varphi}(D-\ell)f(x)
&=\int_{\R^n}e^{i\ell\cdot (x-y)}
\widetilde{\Phi}(x-y)f(y)\, dy
\\
&=\sum_{\widetilde{\ell} \in \Z^n}
\int_{2\pi\widetilde{\ell}+[-\pi,\pi]^n}e^{i\ell\cdot (x-y)}
\widetilde{\Phi}(x-y)f(y)\, dy
\\
&=e^{i \ell \cdot x}
\int_{[-\pi,\pi]^n}e^{-i\ell\cdot y}
\bigg(\sum_{\widetilde{\ell} \in \Z^n}
\widetilde{\Phi}(x-y-2\pi\widetilde{\ell})
f(y+2\pi\widetilde{\ell})\bigg) dy,
\end{align*}
where $\widetilde{\Phi}=\calF^{-1}\widetilde{\varphi}$.
This means that
$|\widetilde{\varphi}(D-\ell)f(x)|$
is equal to $(2 \pi )^n$ times the absolute value of the $\ell$-th Fourier coefficient
of the $(2\pi \Z)^n$-periodic function
$\sum_{\widetilde{\ell} \in \Z^n}
\widetilde{\Phi}(x-y-2\pi\widetilde{\ell})f(y+2\pi\widetilde{\ell})$
of the $y$-variable.
Hence, it follows from Parseval's identity that
\begin{align*}
\sum_{\ell \in \Z^n}
|\widetilde{\varphi}(D-\ell)f(x)|^2
&={(2\pi)^n}
\int_{[-\pi,\pi]^n}
\bigg|\sum_{\widetilde{\ell} \in \Z^n}
\widetilde{\Phi}(x-y-2\pi\widetilde{\ell})
f(y+2\pi\widetilde{\ell})\bigg|^2 dy.
\end{align*}
Since
$\sup_{z \in \R^n}\bigg(\sum_{\widetilde{\ell} \in \Z^n}|
\widetilde{\Phi}(z-2\pi\widetilde{\ell})|\bigg)<\infty$,
by Schwarz's inequality,
the right-hand side of this identity is estimated by
\begin{align*}
\int_{[-\pi,\pi]^n}
\sum_{\widetilde{\ell} \in \Z^n}
|\widetilde{\Phi}(x-y-2\pi\widetilde{\ell})|
|f(y+2\pi\widetilde{\ell})|^2\, dy
=\int_{\R^n}|\widetilde{\Phi}(x-y)||f(y)|^2\, dy.
\end{align*}
Therefore, the rapidly decreasing property of $\widetilde{\Phi}$
gives the desired estimate.
\end{proof}

\begin{lem}\label{basic-estimate-1}
For each $N \in \N_0$ and $\beta,\gamma \in \N_0^n$,
the estimate
\[
|\partial^{\beta}_{\xi}\partial^{\gamma}_{\eta}
\sigma_{j,k,\nu}(x,\xi,\eta)|
\lesssim
2^{jm-kN}2^{-j\rho(|\beta|+|\gamma|)}
\]
holds for $j,k \ge 0$ and $\nu \in \Z^n \times \Z^n$.
\end{lem}
\begin{proof}
First, suppose $k \ge 1$. 
Then by the moment condition of $\calF^{-1}\psi_k$ and 
Taylor's formula, 
we can write \eqref{def-sigmajkn} as 
\begin{align*}
&\sigma_{j,k,\nu}(x,\xi,\eta)
\\
&=2^{j\rho n}
\int_{\R^n}[\calF^{-1}\psi_k](2^{j\rho}y)
\bigg(\sigma_{j,\nu}(x-y,\xi,\eta)
-\sum_{|\alpha|<N}
\frac{\partial_x^{\alpha}\sigma_{j,\nu}(x,\xi,\eta)}{\alpha!}
(-y)^{\alpha}\bigg)dy
\\
&=2^{j\rho n}
\int_{\R^n}
[\calF^{-1}\psi_k](2^{j\rho}y)
\bigg(N\sum_{|\alpha|=N}\frac{(-y)^{\alpha}}{\alpha!}
\int_0^1
(1-t)^{N-1}
[\partial_x^{\alpha}\sigma_{j,\nu}](x-ty,\xi,\eta)\, dt
\bigg)dy.
\end{align*}
Using the fact that $1+|\xi|+|\eta| \approx 2^j$
for $(\xi,\eta) \in \suppast (\sigma_{j,\nu})$, 
we have
\[
|\partial^{\alpha}_x\partial^{\beta}_{\xi}\partial^{\gamma}_{\eta}
\sigma_{j,\nu}(x,\xi,\eta)|
\lesssim 2^{jm+j\rho(|\alpha|-|\beta|-|\gamma|)}.
\]
On the other hand,
it follows from \eqref{dyadic-decomposition} that
\[
|\calF^{-1}\psi_k(y)|
\lesssim 2^{kn}(1+2^k|y|)^{-(N+n+1)}.
\]
Hence 
\begin{align*}
&|\partial^{\beta}_{\xi}\partial^{\gamma}_{\eta}
\sigma_{j,k,\nu}(x,\xi,\eta)|
\\
&\lesssim
\sum_{|\alpha|=N}
2^{j\rho n}
\int_{\R^n}\int_0^1
\bigg|[\calF^{-1}\psi_k](2^{j\rho}y)y^{\alpha}
[\partial_x^{\alpha}\partial^{\beta}_{\xi}\partial^{\gamma}_{\eta}
\sigma_{j,\nu}](x-ty,\xi,\eta)\bigg| dtdy
\\
&\lesssim
2^{(j\rho+k)n}\int_{\R^n}(1+2^{j\rho+k}|y|)^{-(N+n+1)}
|y|^{N}2^{jm+j\rho(N-|\beta|-|\gamma|)}\, dy
\\
&
\approx 
2^{jm-kN}2^{-j\rho(|\beta|+|\gamma|)}. 
\end{align*}
If $k=0$, then using \eqref{def-sigmajkn} 
and slightly modifying the above argument, 
we obtain the desired estimate. 
\end{proof}

\begin{lem}\label{basic-estimate-2}
For each $N \in \N_0$,
the estimate
\[
|T_{\sigma_{j,k,\nu}}(f,g)(x)|
\lesssim
2^{jm-kN}S_{2^{j\rho}}(f)(x)S_{2^{j\rho}}(g)(x)
\]
holds for $j,k \ge 0$ and $\nu \in \Z^n \times \Z^n$.
\end{lem}
\begin{proof}
We write
\[
T_{\sigma_{j,k,\nu}}(f,g)(x)
=\int_{\R^n \times \R^n}
K_{j,k,\nu}(x,x-y,x-z)f(y)g(z)\, dydz,
\]
where
\[
K_{j,k,\nu}(x,y,z)
=\frac{1}{(2\pi)^{2n}}\int_{\R^n \times \R^n}
e^{i(y\cdot\xi+z\cdot\eta)}
\sigma_{j,k,\nu}(x,\xi,\eta)\, d\xi d\eta.
\]
Since
$|\xi-2^{j\rho}\nu_1| \lesssim 2^{j\rho}$
and
$|\eta-2^{j\rho}\nu_2| \lesssim 2^{j\rho}$
for $(\xi,\eta) \in \suppast (\sigma_{j,k,\nu})$,
it follows from Lemma \ref{basic-estimate-1}
and integration by parts that
\[
|K_{j,k,\nu}(x,y,z)|
\lesssim
2^{jm-kN}
\frac{2^{j\rho n}}{(1+2^{j\rho}|y|)^{n+1}}
\frac{2^{j\rho n}}{(1+2^{j\rho}|z|)^{n+1}}.
\]
From this the desired estimate follows.
\end{proof}

The estimate 
\begin{equation}\label{single-estimate}
\|T_{\sigma_{j,k,\nu}}(f,g)\|_{L^2}
\lesssim 
2^{jm -kN} \|f\|_{L^2} \|g\|_{L^{\infty}}
\end{equation}
immediately follows from Lemma \ref{basic-estimate-2}. 
In the lemmas below, we shall derive finer 
estimates by utilizing orthogonality.

\begin{lem}\label{sum-on-line}
{\rm (1)} For each $N \in \N_0$, the estimate 
\[
\bigg\|\sum_{\nu_2 \in \Z^n}
T_{\sigma_{j,k,\nu}}(f,g)\bigg\|_{L^2}
\lesssim 2^{jm-kN}\|f\|_{L^2}\|g\|_{L^\infty}
\]
holds for all $j,k \ge 0$ and all $\nu_1 \in \Z^n$. 
\\
{\rm (2)} For each $N \in \N_0$, the estimate 
\[
\bigg\|\sum_{\nu_1+\nu_2=\mu}
T_{\sigma_{j,k,\nu}}(f,g)\bigg\|_{L^2}
\lesssim 2^{jm-kN}\|f\|_{L^2}\|g\|_{L^\infty}
\]
holds for all $j,k \ge 0$ and all $\mu \in \Z^n$. 
\end{lem}
\begin{proof} 
Proof of (1). 
Take a function 
$\widetilde{\varphi}\in C_{0}^{\infty}(\R^n)$ 
such that 
$\widetilde{\varphi} (\xi)=1$ on $\supp \varphi$. 
Then, by \eqref{replace-f1-g1}, 
\begin{equation}\label{replace-fjn-gjn}
T_{\sigma_{j,k,\nu}} (f,g)
=T_{\sigma_{j,k,\nu}} (f_{j,\nu_1},g)
=T_{\sigma_{j,k,\nu}} (f,g_{j,\nu_2})  
\end{equation}
with 
\begin{equation}\label{fjn-gjn}
f_{j,\nu_1}=\widetilde{\varphi} (2^{-j\rho}D - \nu_1)f, 
\quad 
g_{j,\nu_2}=\widetilde{\varphi} (2^{-j\rho}D - \nu_2)g.
\end{equation}
From \eqref{FouriersuppTsjkn}, 
we see that 
\begin{equation}\label{Fouriersupp-n1n2}
\supp \calF [T_{\sigma_{j,k,\nu}}(f,g)] 
\subset 
2^{j\rho}(\nu_1 + \nu_2) + 2^{j\rho + k+2}Q.  
\end{equation}
Notice that for fixed $\nu_1 \in \Z^{n}$  
the interaction of the family 
$\{2^{j\rho}(\nu_1 + \nu_2) 
+ 2^{j\rho + k+2}Q\}_{\nu_2 \in\Z^n}$ 
is $\lesssim 2^{kn}$.  
Hence, by Lemma \ref{orthogonality} (1) and 
Lemma \ref{basic-estimate-2}, 
we have 
\begin{align}
&
\bigg\|
\sum_{\nu_2 \in \Z^n}
T_{\sigma_{j,k, \nu}}
(f,g)
\bigg\|_{L^2}^2
=
\bigg\|
\sum_{\nu_2 \in \Z^n}
T_{\sigma_{j,k, \nu}}
(f,g_{j,\nu_2})
\bigg\|_{L^2}^2
\nonumber
\\
&
\lesssim 
2^{kn}
\sum_{\nu_2 \in \Z^n}
\|
T_{\sigma_{j,k, \nu}}(f,g_{j,\nu_2})
\|_{L^2}^2
\nonumber 
\\
&
\lesssim 
2^{kn + 2(jm-kN)}
\sum_{\nu_2 \in \Z^n}
\int_{\R^n}
S_{2^{j\rho}}(f)(x)^2
S_{2^{j\rho}}(g_{j,\nu_2})(x)^2
\, dx. 
\label{1111}
\end{align}
By Schwarz's inequality and 
Lemma \ref{basic-estimate-0}, 
\begin{align*}
&
\sum_{\nu_2 \in \Z^n}
\int_{\R^n}
S_{2^{j\rho}}(f)(x)^2
S_{2^{j\rho}}(g_{j,\nu_2})(x)^2
\, dx
\\
&
\lesssim 
\sum_{\nu_2 \in \Z^n}
\int_{\R^n}
S_{2^{j\rho}}(f^2)(x)
S_{2^{j\rho}}(g_{j,\nu_2}^2)(x)
\, dx
\\
&
\lesssim 
\int_{\R^n}
S_{2^{j\rho}}(f^2)(x)
S_{2^{j\rho}}\big[S_{2^{j\rho}}(g^2)\big](x)
\, dx
\\
&
\lesssim 
\|f\|_{L^2}^2
\|g\|_{L^{\infty}}^2.
\end{align*}
Since $N$ can be taken arbitrarily large, 
we obtain the desired estimate.

Proof of (2). 
By \eqref{replace-fjn-gjn}-\eqref{fjn-gjn}, 
Lemma \ref{basic-estimate-2}, and Schwarz's inequality, 
we have 
\begin{align*}
&
\bigg|
\sum_{\nu_1 + \nu_2=\mu} 
T_{\sigma_{j,k,\nu}}
(f,g)(x)
\bigg|
\le 
\sum_{\nu_1 + \nu_2=\mu} 
|T_{\sigma_{j,k,\nu}}
(f_{j,\nu_1},g_{j,\nu_2})(x)|
\\
&
\lesssim 
2^{jm-kN}
\sum_{\nu_1 + \nu_2=\mu} 
S_{2^{j\rho}}(f_{j,\nu_1})(x)
S_{2^{j\rho}}(g_{j,\nu_2})(x)
\\
&
\le 
2^{jm-kN}
\bigg(\sum_{\nu_1} 
S_{2^{j\rho}}(f_{j,\nu_1})(x)^2
\bigg)^{1/2}
\bigg(\sum_{\nu_1} 
S_{2^{j\rho}}(g_{j,\mu -\nu_1})(x)^2
\bigg)^{1/2}
\\
&
\lesssim  
2^{jm-kN}
\bigg(\sum_{\nu_1} 
S_{2^{j\rho}}(f_{j,\nu_1}^2)(x)
\bigg)^{1/2}
\bigg(\sum_{\nu_1} 
S_{2^{j\rho}}(g_{j,\mu -\nu_1}^2)(x)
\bigg)^{1/2}. 
\end{align*}
By Lemma \ref{basic-estimate-0}, 
we have 
\begin{equation*}
\bigg(\sum_{\nu_1} 
S_{2^{j\rho}}(f_{j,\nu_1}^2)(x)
\bigg)^{1/2}
\lesssim 
S_{2^{j\rho}}\big[S_{2^{j\rho}}(f^2)\big]
(x)^{1/2}
\approx 
S_{2^{j\rho}}(f^2)(x)^{1/2}.
\end{equation*}
Similarly, 
\[
\bigg(\sum_{\nu_1} 
S_{2^{j\rho}}(g_{j,\mu -\nu_1}^2)(x)
\bigg)^{1/2}
\lesssim 
S_{2^{j\rho}}(g^2)(x)^{1/2}
\lesssim \|g\|_{L^{\infty}}. 
\] 
Thus we have the pointwise estimate 
\[
\bigg|
\sum_{\nu_1 + \nu_2=\mu} 
T_{\sigma_{j,k,\nu}}
(f,g)(x)
\bigg|
\lesssim 
2^{jm-kN} 
S_{2^{j\rho}}(f^2)(x)^{1/2}
\|g\|_{L^{\infty}}, 
\]
from which the desired $L^2$ inequality follows. 
\end{proof}

\begin{lem}\label{sum-on-several-lines} 
{\rm (1)} For each $N \in \N_0$, the estimate 
\[
\bigg\|
\sum_{\nu_1 \in \Lambda}\sum_{\nu_2 \in \Z^n}
T_{\sigma_{j,k,\nu}}(f,g)\bigg\|_{L^2}
\lesssim 
|\Lambda|^{1/2}
2^{jm-kN}\|f\|_{L^2}\|g\|_{L^\infty}
\]
holds for all $j,k \ge 0$ and all finite sets $\Lambda \subset \Z^n$. 
\\
{\rm (2)} For each $N \in \N_0$, the estimate 
\[
\bigg\|\sum_{\mu \in \Lambda}\sum_{\nu_1+\nu_2=\mu}
T_{\sigma_{j,k,\nu}}(f,g)\bigg\|_{L^2}
\lesssim 
|\Lambda|^{1/2}
2^{jm-kN}\|f\|_{L^2}\|g\|_{L^\infty}
\]
holds for all $j,k \ge 0$ and all finite sets $\Lambda \subset \Z^n$. 
\end{lem}
\begin{proof} 
For the proof of (1) and (2),
we freeze $g \in L^{\infty}$ 
and consider the linear operator $T_{\sigma_{j,k,\nu}}(\cdot,g)$ 
defined by $[T_{\sigma_{j,k,\nu}}(\cdot,g)](f)=
T_{\sigma_{j,k,\nu}}(f,g)$
for $j,k \ge 0$ and $\nu \in \Z^n \times \Z^n$. 
By \eqref{single-estimate}, 
$T_{\sigma_{j,k,\nu}}(\cdot,g)$ is a bounded linear operator in $L^2$.

Proof of (1). 
Since 
\[
\suppast \bigg(\sum_{\nu_2\in Z^n} \sigma_{j,k,\nu}\bigg)
\subset 
\supp \varphi (2^{-j\rho}\cdot -\nu_1) \times \R^n,
\]
we have
\[
\sum_{\nu_2 \in \Z^n}T_{\sigma_{j,k,\nu}}(f,g)
=\sum_{\nu_2 \in \Z^n}T_{\sigma_{j,k,\nu}}
(\ichi_{\supp \varphi (2^{-j\rho}\cdot -\nu_1)}(D)f,g).
\]
In terms of the linear operator, this can be written as 
\[
\sum_{\nu_2 \in \Z^n}
T_{\sigma_{j,k,\nu}}(\cdot,g)=
\bigg(\sum_{\nu_2 \in \Z^n}
T_{\sigma_{j,k,\nu}}(\cdot,g)\bigg)
\ichi_{\supp \varphi (2^{-j\rho}\cdot -\nu_1)}(D).
\]
Since the interaction of the family 
$\{\supp \varphi (2^{-j\rho}\cdot -\nu_1) \}_{\nu_{1}}$ 
is $\lesssim 1$,
we see that 
\[
\bigg|\bigg\{\widetilde{\nu_1} \in \Z^n \,:\,
\bigg(\sum_{\nu_2 \in \Z^n}
T_{\sigma_{j,k,(\widetilde{\nu_1},\nu_2)}}(\cdot,g)\bigg)
\bigg(\sum_{\nu_2 \in \Z^n}
T_{\sigma_{j,k,(\nu_1,\nu_2)}}(\cdot,g)\bigg)^{\ast} \neq 0
\bigg\}\bigg|
\lesssim 1
\]
for all $\nu_1 \in \Z^n$. 
Thus Lemma \ref{orthogonality} (3) yields 
\[
\bigg\|
\sum_{\nu_1 \in \Lambda}
\sum_{\nu_2 \in \Z^n}
T_{\sigma_{j,k,\nu}}(\cdot ,g)
\bigg\|_{L^2\to L^2}^2
\lesssim 
\sum_{\nu_1 \in \Lambda}
\bigg\|
\sum_{\nu_2 \in \Z^n}
T_{\sigma_{j,k,\nu}}(\cdot ,g)
\bigg\|_{L^2\to L^2}^2. 
\]
By Lemma \ref{sum-on-line} (1), the 
right-hand side of the above is $\lesssim 2^{2(jm-kN)}
|\Lambda|\|g\|_{L^{\infty}}^2$, 
which implies the desired estimate.

Proof of (2).  
As in \eqref{Fouriersupp-n1n2}, 
the formula \eqref{FouriersuppTsjkn} implies 
\[
\supp \calF \bigg[\sum_{\nu_1+\nu_2=\mu }
T_{\sigma_{j,k,\nu}}(f,g)
\bigg] 
\subset 
2^{j\rho}\mu 
+ 2^{j\rho+ k+2}Q,
\]
which gives
\[
\sum_{\nu_1+\nu_2=\mu}
T_{\sigma_{j,k,\nu}}(f,g)
=\ichi_{2^{j\rho}\mu 
+ 2^{j\rho+ k+2}Q}(D)
\bigg(\sum_{\nu_1+\nu_2=\mu}
T_{\sigma_{j,k,\nu}}(f,g)\bigg).
\]
Thus, in terms of the linear operator, 
\[
\sum_{\nu_1+\nu_2=\mu}
T_{\sigma_{j,k,\nu}}(\cdot,g)
=\ichi_{2^{j\rho}\mu 
+ 2^{j\rho+ k+2}Q}(D)
\bigg(\sum_{\nu_1+\nu_2=\mu}
T_{\sigma_{j,k,\nu}}(\cdot,g)\bigg).
\]
Since the interaction of the family 
$\{2^{j\rho}\mu 
+ 2^{j\rho+ k+2}Q\}_{\mu \in \Z^n}$ 
is $\lesssim 2^{kn}$,
we see that 
\[
\bigg|\bigg\{\widetilde{\mu} \in \Z^n \,:\,
\bigg(\sum_{\nu_1+\nu_2=\widetilde{\mu}}
T_{\sigma_{j,k,\nu}}(\cdot,g)\bigg)^{\ast}
\bigg(\sum_{\nu_1+\nu_2=\mu}
T_{\sigma_{j,k,\nu}}(\cdot,g)\bigg)
\neq 0
\bigg\}\bigg|
\lesssim 2^{kn}
\]
for all $\mu \in \Z^n$. 
Hence Lemma \ref{orthogonality} (2) yields 
\[
\bigg\|
\sum_{\mu \in \Lambda}
\sum_{\nu_1 + \nu_2 =\mu}
T_{\sigma_{j,k,\nu}}(\cdot ,g)
\bigg\|_{L^2\to L^2}^2
\lesssim 
2^{kn}
\sum_{\mu \in \Lambda}
\bigg\|
\sum_{\nu_1 + \nu_2 =\mu}
T_{\sigma_{j,k,\nu}}(\cdot ,g)
\bigg\|_{L^2\to L^2}^2. 
\]
By Lemma \ref{sum-on-line} (2), the 
right-hand side of the above is 
$\lesssim 2^{kn+ 2(jm-kN)}|\Lambda|\|g\|_{L^{\infty}}^2$. 
Since $N$ can be taken arbitrarily large, we obtain the desired estimate. 
\end{proof}

Notice that 
$\sigma_{j,k,\nu}\neq 0$ only for $|\nu_{1}|\lesssim 2^{j(1-\rho)}$ 
and $|\nu_2|\lesssim 2^{j(1-\rho)}$ and hence 
$\sum_{\nu \in \Z^n \times \Z^n}\, T_{\sigma_{j,k,\nu}}$ 
can be written as the 
sum of Lemma \ref{sum-on-several-lines} (1) or (2) 
with $|\Lambda|\approx 2^{j(1-\rho)n}$. 
Hence the following lemma 
directly follows from Lemma \ref{sum-on-several-lines}.

\begin{lem}\label{sum-on-annulus} 
If $m=-(1-\rho)n/2$ and $0\le \rho\le 1$, 
then for each $N \in \N_0$ the estimate 
\[
\bigg\|\sum_{\nu \in \Z^n \times \Z^n}\, T_{\sigma_{j,k,\nu}}
(f,g)\bigg\|_{L^2}
\lesssim 
2^{-kN}\|f\|_{L^2}\|g\|_{L^\infty}
\]
holds for $j,k \ge 0$. 
\end{lem}

\subsection{Proof of Theorem \ref{main-thm-1}}\label{subsection-proof}

Throughout this subsection, we assume $m$, $\rho$, and $\sigma$ 
satisfy the conditions of Theorem \ref{main-thm-1}, namely, 
$0\le \rho <1$, $m=-(1-\rho)n/2$, 
and $\sigma \in BS^{m}_{\rho,\rho}(\R^n)$.

Before proceeding to the main argument, 
we shall see that it is sufficient to consider 
the case where $\suppast \sigma$ 
is included in a cone minus a ball centered at the origin.

To see this, take a function 
$\Theta \in C_{0}^{\infty}(\R^n \times \R^n)$ 
such that 
$\Theta (\xi, \eta)=1$ on 
$\{(|\xi|^2+ |\eta|^2)^{1/2}\le 2\}$ 
and $\supp \Theta 
\subset \{(|\xi|^2+ |\eta|^2)^{1/2}\le 4\}$, 
and write $\sigma$ as 
\[
\sigma(x, \xi,\eta)=\sigma(x, \xi,\eta)\Theta(\xi,\eta)
+\sigma(x, \xi,\eta)(1-\Theta(\xi,\eta)).
\]
By simply summing the estimate 
of Lemma \ref{sum-on-annulus} over $k\ge 0$ and $0\le j\le 2$, 
we obtain   
\[
\|T_{\sigma\Theta}(f,g)\|_{L^2}
\lesssim \|f\|_{L^2}\|g\|_{L^{\infty}}.  
\]
Hence 
it is sufficient to treat only $T_{\sigma(1-\Theta)}$. 
Next, if $(\xi,\eta)$ belongs to the unit 
sphere $\Sigma$ of $\R^n \times \R^n$, 
then either $\xi+\eta \neq 0$ or $\xi \neq 0$. 
By the compactness of $\Sigma$,
this implies that there exists a constant $c>0$ such that
$\Sigma$ is covered by the two open sets
\begin{equation*}
V_1=\{(\xi,\eta) \in \Sigma \,:\,
|\xi+\eta|>c\},
\quad 
V_2=\{(\xi,\eta) \in \Sigma \,:\,
|\xi|>c\}.
\end{equation*}
Taking a smooth partition of unity $\Phi_i$, $i=1,2$,
on $\Sigma$ such that $\mathrm{supp}\, \Phi_i \subset V_i$,
we decompose $\sigma (1-\Theta)$ as 
\[
\sigma(x,\zeta)(1-\Theta(\zeta))
=\sum_{i=1}^{2}\sigma(x,\zeta)(1-\Theta(\zeta))
\Phi_i(\zeta/|\zeta|)
=\sum_{i=1}^{2}\sigma^{(i)}(x,\zeta),
\quad \zeta=(\xi,\eta).
\]
It is sufficient to prove the estimate 
for each $T_{\sigma^{(i)}}$, $i=1,2$. 
Obviously $\sigma^{(i)} \in BS^m_{\rho,\rho}(\R^n)$.

To sum up, by writing $\sigma^{(i)}$ simply as $\sigma$,
we may assume that $\sigma$ satisfies the additional condition 
\[
\suppast \sigma
\subset \Gamma(V_i)=\{\zeta \in \R^{2n} 
\,:\, \zeta/|\zeta| \in V_i, \ |\zeta| \ge 2\}
\quad
\text{for $i=1$ or $2$}.
\]
For such $\sigma$, we have 
$\sigma_{j,k,\nu}=0$ for $j=0$ 
and thus the decomposition \eqref{symbol-dec-0} 
takes the form 
\begin{equation}\label{symbol-dec-1}
\sigma = 
\sum_{j\ge 1}\,
\sum_{\nu =(\nu _1, \nu_2)\in \Z^n \times \Z^n}\,
\sigma_{j,\nu}
=
\sum_{j\ge 1}\sum_{k\ge 0}\,
\sum_{\nu =(\nu _1, \nu_2)\in \Z^n \times \Z^n}\,
\sigma_{j,k,\nu}. 
\end{equation}
In the rest of the proof, we shall consider the two cases 
\[
\suppast \sigma \subset \Gamma(V_1), 
\quad 
\suppast \sigma \subset \Gamma(V_2) 
\]
separately.

We shall prove the following 
estimate for the trilinear form: 
\[
|\langle T_{\sigma}(f,g), h \rangle |
\lesssim 
\|f\|_{L^2} 
\|g\|_{L^{\infty}} 
\|h\|_{L^2},   
\]
which is equivalent to the desired estimate 
for the operator $T_{\sigma}$. 
\vspace{12pt}

\noindent {\bf The case $\suppast \sigma \subset \Gamma(V_1)$}.

In this case, all $(\xi, \eta) \in \suppast \sigma$ satisfy 
$|\xi+ \eta|\approx (|\xi|^2 + |\eta|^2)^{1/2}$ 
(but $|\xi|$ may be small compared with $(|\xi|^2 + |\eta|^2)^{1/2}$). 
We take a positive integer $a$ such that 
\begin{equation}\label{suppsigma-xi+eta}
 (\xi, \eta)\in \suppast (\sigma_{j,k,\nu}) 
\, \Rightarrow \, 
2^{j-a}\le |\xi + \eta|\le 2^{j+a}. 
\end{equation}
Using this $a$, we write \eqref{symbol-dec-1} as 
\begin{equation*}
\sigma= \sum_{j \ge 1}\sum_{k \ge 0}\sum_{\nu}
\sigma_{j,k,\nu}
=
\sum_{\substack{
j \ge 1, \, k \ge 0 
\\ 
k \le j(1-\rho)-a-2}}
\sum_{\nu}
\sigma_{j,k,\nu}
+\sum_{\substack{
j \ge 1, \, k \ge 0 
\\ 
k > j(1-\rho)-a-2}}
\sum_{\nu}
\sigma_{j,k,\nu}.
\end{equation*}
According to this decomposition of 
$\sigma$, we write 
the trilinear form as 
\begin{align*}
&\langle T_{\sigma}(f,g), h \rangle 
\\
&=
\sum_{\substack{
j \ge 1, \, k \ge 0 
\\ 
k \le j(1-\rho)-a-2}}
\sum_{\nu}
\langle T_{\sigma_{j,k,\nu}}(f,g), h \rangle 
+\sum_{\substack{
j \ge 1, \, k \ge 0 
\\ 
k > j(1-\rho)-a-2}}
\sum_{\nu}
\langle T_{\sigma_{j,k,\nu}}(f,g), h \rangle
\\
&=X_1 + X_2, 
\quad 
\text{say}. 
\end{align*}

The estimate for the second term $X_2$ is easy. 
In fact, Lemma \ref{sum-on-annulus} gives 
\begin{align*}
&|X_2|
\le 
\sum_{\substack{
j \ge 1, \, k \ge 0 
\\ 
k > j(1-\rho)-a-2}}
\bigg\|
\sum_{\nu}
T_{\sigma_{j,k,\nu}}(f,g)
\bigg\|_{L^2}
\|h\|_{L^2}
\\
&
\lesssim 
\sum_{\substack{
j \ge 1, \, k \ge 0 
\\ 
k > j(1-\rho)-a-2}}
2^{-kN}
\|f\|_{L^2}
\|g\|_{L^{\infty}}
\|h\|_{L^2}
\approx 
\|f\|_{L^2}
\|g\|_{L^{\infty}}
\|h\|_{L^2},  
\end{align*}
where we used the assumption $\rho<1$ in the last $\approx$.

In order to estimate $X_1$, we use the decomposition 
\[
f=\sum_{\ell } f_{\ell}, 
\quad 
f_{\ell}=\psi_{\ell}(D)f, 
\]
and write 
\[
X_{1}
=
\sum_{\substack{
j \ge 1, \, k \ge 0 
\\ 
k \le j(1-\rho)-a-2}}
\sum_{\ell \ge 0}
\sum_{\nu}
\langle T_{\sigma_{j,k,\nu}}(f_{\ell},g), h \rangle. 
\]

Here we make simple observations.  
First, if $k\le j(1-\rho) -a-2$, then 
from \eqref{FouriersuppTsjkn} and \eqref{suppsigma-xi+eta} 
we see that 
\begin{align*}
\supp \calF [T_{\sigma_{j,k,\nu}}(f,g)] 
&\subset 
\bigcup_{2^{j-a} \le |\xi+\eta|\le 2^{j+a}}
\{\zeta \in \R^n 
\,:\, 
|\zeta - \xi -\eta|\le 2^{j\rho + k +1}
\}
\\
&\subset 
\{\zeta \in \R^n 
\,:\, 
2^{j-a-1}\le |\zeta|\le 2^{j+a+1}
\}. 
\end{align*}
Hence, by \eqref{replace-h1}, 
$\langle T_{\sigma_{j,k,\nu}}(f_{\ell},g), h \rangle $ 
in $X_1$ can be written as 
\[
\langle T_{\sigma_{j,k,\nu}}(f_{\ell},g), h \rangle 
=\langle T_{\sigma_{j,k,\nu}}(f_{\ell},g), h_{j} \rangle, 
\quad 
 h_j =\theta (2^{-j} D)h, 
\]
where 
$\theta$ is an appropriate 
function supported in an annulus. 
Secondly, 
since $\supp \widehat{f}_{\ell}
\subset \{2^{\ell -1}\le |\xi|\le 2^{\ell +1}\}$ 
for $\ell >0$ and since 
$\suppast (\sigma_{j,k,\nu}) \subset \{|\xi|\le 2^{j+1}\}$, 
it follows that 
$T_{\sigma_{j,k,\nu}}(f_{\ell},g)=0$ for 
$\ell > j+1$. 
Thirdly, since 
$\suppast (\sigma_{j,k,\nu})
\subset 
\supp \varphi (2^{-j \rho} \cdot - \nu_{1})
\times \R^n$ 
and $\supp \widehat{f}_{\ell} \subset \supp \psi_{\ell}$, 
we have 
\[
T_{\sigma_{j,k,\nu}}(f_{\ell},g)\neq 0
\; \Rightarrow \; 
\supp \varphi (2^{- j\rho}\cdot - \nu_1) 
\cap \supp \psi_{\ell} \neq 
\emptyset . 
\]
Combining these observations, we see that $X_1$ can be written as 
\begin{equation}\label{X1}
X_1 
=
\sum_{\substack{
j \ge 1, \, k \ge 0 
\\ 
k \le j(1-\rho)-a-2}}
\,
\sum_{\ell =0}^{j+1}
\, 
\sum_{\nu_{1}\in \Lambda_{j,\ell}}
\, 
\sum_{\nu_2 \in \Z^n}
\, 
\langle T_{\sigma_{j,k,\nu}}(f_{\ell},g), h_j \rangle, 
\end{equation}
where 
\[
\Lambda_{j,\ell}
=\{\nu_1 \in \Z^n
\,:\,
\supp \varphi (2^{-j\rho}\cdot - \nu_1) 
\cap \supp \psi_{\ell}
\neq \emptyset
\}. 
\]

The number of elements of $\Lambda_{j,\ell}$ satisfies 
\[
|\Lambda_{j,\ell}|
\lesssim 
(\max \{1, 2^{\ell - j\rho}\})^{n}. 
\]
Thus Lemma \ref{sum-on-several-lines} (1) gives 
\begin{align*}
&\bigg|
\sum_{\nu_{1}\in \Lambda_{j,\ell}}
\sum_{\nu_2 \in \Z^n}
\langle T_{\sigma_{j,k,\nu}}(f_{\ell},g), h_j \rangle
\bigg|
\\
&
\le 
\bigg\|
\sum_{\nu_{1}\in \Lambda_{j,\ell}}
\sum_{\nu_2 \in \Z^n}
T_{\sigma_{j,k,\nu}}(f_{\ell},g)
\bigg\|_{L^2}
\|h_j \|_{L^2}
\\
&
\lesssim 
\max \{1, 2^{(\ell -j\rho)n/2}\}
2^{jm -kN} 
\|f_{\ell}\|_{L^2}
\|g\|_{L^{\infty}}
\|h_j \|_{L^2}. 
\end{align*}
Hence 
\begin{align}
\label{case1-final}
|X_1 |
&\lesssim 
\sum_{\substack{
j \ge 1, \, k \ge 0 
\\ 
k \le j(1-\rho)-a-2}}
\sum_{\ell =0}^{j+1}
\max \{1, 2^{(\ell -j\rho)n/2}\}
2^{jm -kN} 
\|f_{\ell}\|_{L^2}
\|g\|_{L^{\infty}}
\|h_j \|_{L^2}
\nonumber 
\\
&\le 
\sum_{k\ge 0}\, 
\sum_{\substack{
j\ge 1, \, \ell \ge 0 
\\
\ell \le j+1
}}
\max \{1, 2^{(\ell -j\rho)n/2}\}
2^{jm -kN} 
\|f_{\ell}\|_{L^2}
\|g\|_{L^{\infty}}
\|h_j \|_{L^2}.
\end{align}
Under our assumption 
$m=-(1-\rho)n/2<0$, it holds that 
\begin{equation}\label{Schur1}
\sum_{j \ge 1}
\ichi \{\ell \le j+1\}
\max\{1,2^{(\ell-j\rho)n/2}\}\, 2^{jm}
\approx 1 
\quad \text{for all}\quad 
\ell \ge 0
\end{equation}
and 
\begin{equation}\label{Schur2}
\sum_{\ell \ge 0}
\ichi \{\ell \le j+1\}
\max\{1,2^{(\ell-j\rho)n/2}\}\, 2^{jm}
\approx 1
\quad \text{for all}\quad  
j\ge 1. 
\end{equation}
Hence, by Schur's lemma, 
\eqref{case1-final} is bounded by 
\[
\sum_{k\ge 0}
2^{-kN} 
\bigg(
\sum_{\ell \ge 0} \|f_{\ell}\|_{L^2}^2
\bigg)^{1/2}
\bigg(
\sum_{j\ge 1} \|h_{j}\|_{L^2}^2
\bigg)^{1/2}
\|g\|_{L^{\infty}}
\lesssim 
\|f\|_{L^2} 
\|h\|_{L^2} 
\|g\|_{L^{\infty}}.  
\]
This completes the proof for the first case. 
\vspace{12pt}

\noindent {\bf The case $\suppast \sigma 
\subset \Gamma(V_2)$}.

In this case, 
all $(\xi, \eta) \in \suppast \sigma$ satisfy 
$|\xi| \approx (|\xi|^2 + |\eta|^2)^{1/2}$ 
(but $|\xi + \eta|$ may be small compared 
with $(|\xi|^2 + |\eta|^2)^{1/2}$). 
We divide the sum over $(j,k)$ in \eqref{symbol-dec-1} into 
two parts $k \le j(1-\rho)$ and $k > j(1-\rho)$ and write 
the trilinear form $\langle T_{\sigma}(f,g), h \rangle$ as 
\begin{align*}
\langle T_{\sigma}(f,g), h \rangle 
&=
\sum_{\substack{
j \ge 1, \, k \ge 0 
\\ 
k \le j(1-\rho)}}
\, 
\sum_{\nu}
\, 
\langle T_{\sigma_{j,k,\nu}}(f,g), h \rangle 
+
\sum_{\substack{
j \ge 1, \, k \ge 0 
\\ 
k > j(1-\rho)}}
\, 
\sum_{\nu}
\, 
\langle T_{\sigma_{j,k,\nu}}(f,g), h \rangle
\\
&=Y_1 + Y_2, 
\quad 
\text{say}. 
\end{align*}

As in the first case, 
the estimate for the second term $Y_2$ is easy. 
In fact, Lemma \ref{sum-on-annulus} gives 
\begin{align*}
&|Y_2|
\le 
\sum_{\substack{
j \ge 1, \, k \ge 0 
\\ 
k > j(1-\rho)}}
\bigg\|
\sum_{\nu}
T_{\sigma_{j,k,\nu}}(f,g)
\bigg\|_{L^2}
\|h\|_{L^2}
\\
&
\lesssim 
\sum_{\substack{
j \ge 1, \, k \ge 0 
\\ 
k > j(1-\rho)}}
2^{-kN}
\|f\|_{L^2}
\|g\|_{L^{\infty}}
\|h\|_{L^2}
\approx 
\|f\|_{L^2}
\|g\|_{L^{\infty}}
\|h\|_{L^2},  
\end{align*}
where the last $\approx$ holds because $1-\rho>0$.

In order to estimate $Y_1$, we use the decomposition 
\[
h=\sum_{\ell } h_{\ell}, 
\quad 
h_{\ell}=\psi_{\ell}(D)h, 
\]
and write 
\[
Y_{1}
=
\sum_{\substack{
j \ge 1, \, k \ge 0 
\\ 
k \le j(1-\rho)}}\, 
\sum_{\ell \ge 0}\, 
\sum_{\nu = (\nu_1, \nu_2)\in \Z^n \times \Z^n}\, 
\langle T_{\sigma_{j,k,\nu}}(f,g), h_{\ell} \rangle . 
\]

Here observe the following. 
Firstly, since 
$|\xi|\approx (|\xi|^2 + |\eta|^2)^{1/2}$ 
for $(\xi, \eta) \in \suppast \sigma$, 
there exists a positive integer $b$ such that 
$\suppast (\sigma_{j,k,\nu}) \subset \{(\xi, \eta) 
\,:\, 
2^{j-b}\le |\xi| \le 2^{j+b}\}$. 
Hence, by \eqref{replace-f1-g1}, 
\[
T_{\sigma_{j,k,\nu}}(f,g)
= T_{\sigma_{j,k,\nu}}(f_j,g), 
\quad 
 f_j =\theta (2^{-j} D)f, 
\]
where 
$\theta$ is an appropriate 
function supported in an annulus. 
Secondly, if $k\le j(1-\rho)$, then 
\eqref{FouriersuppTsjkn} yields 
\begin{align*}
\supp \calF [T_{\sigma_{j,k,\nu}}(f_j,g)] 
&\subset 
\bigcup_{(|\xi|^2+|\eta|^2)^{1/2}\le 2^{j+1}} 
\{\zeta \in \R^n 
\,:\, 
|\zeta - \xi -\eta|\le 2^{j\rho + k +1}
\}
\\
&\subset 
\{\zeta \in \R^n 
\,:\, 
|\zeta|\le 2^{j+3}
\}, 
\end{align*}
which together with the fact 
$\supp \widehat{h}_{\ell}\subset \supp \psi_{\ell}$ 
implies 
$\langle T_{\sigma_{j,k,\nu}}(f_j,g), h_{\ell}\rangle =0$ 
for 
$\ell > j+3$.
Thirdly, as we have already seen, 
\eqref{Fouriersupp-n1n2} holds, and hence, 
by \eqref{replace-h1}, 
\begin{equation*}
\langle T_{\sigma_{j,k,\nu}}(f_j,g), h_{\ell} \rangle \ne 0
\;\Rightarrow \;
(2^{j\rho} (\nu_1 +\nu_2) + 2^{j\rho + k+2} Q )  
\cap \supp \psi_{\ell} \neq 
\emptyset. 
\] 
Combining these observations, we see that $Y_1$ can be written as 
\begin{equation}\label{Y1}
Y_1 
=
\sum_{\substack{
j \ge 1, \, k \ge 0 
\\ 
k \le j(1-\rho)}}
\, 
\sum_{\ell =0}^{j+3}\, 
\sum_{\mu \in \Lambda_{j,k,\ell}}\, 
\sum_{\nu_{1}+ \nu_2 = \mu}\, 
\langle T_{\sigma_{j,k,\nu}}(f_j,g), h_{\ell} \rangle, 
\end{equation}
where 
\[
\Lambda_{j,k,\ell}
=\{\mu \in \Z^n
\,:\,
(2^{j\rho} \mu + 2^{j\rho + k+2} Q )  
\cap \supp \psi_{\ell} \neq 
\emptyset 
\}. 
\]

The number of elements of $\Lambda_{j,k,\ell}$ is 
estimated by 
\begin{equation*}
|\Lambda_{j,k,\ell}|
\lesssim 
(\max \{2^{k}, 2^{\ell -j \rho}\})^{n} 
\lesssim 
2^{kn} 
\max \{1, 2^{(\ell -j \rho)n}\}. 
\end{equation*}
Thus Lemma \ref{sum-on-several-lines} (2) gives 
\begin{align*}
&\bigg|
\sum_{\mu \in \Lambda_{j,k,\ell}}
\sum_{\nu_1+ \nu_2 =\mu }
\langle T_{\sigma_{j,k,\nu}}(f_{\ell},g), h_j \rangle
\bigg|
\\
&
\le 
\bigg\|
\sum_{\mu \in \Lambda_{j,k,\ell}}
\sum_{\nu_1+ \nu_2 =\mu }
T_{\sigma_{j,k,\nu}}(f_{\ell},g)
\bigg\|_{L^2}
\|h_j \|_{L^2}
\\
&
\lesssim 
\max \{1, 2^{(\ell -j\rho)n/2}\}
2^{jm -k(N-n/2)} 
\|f_{\ell}\|_{L^2}
\|g\|_{L^{\infty}}
\|h_j \|_{L^2}. 
\end{align*}
Hence 
\begin{align}
\label{case2-final}
|Y_1 |
&\lesssim 
\sum_{\substack{
j \ge 1, \, k \ge 0 
\\ 
k \le j(1-\rho)}}
\sum_{\ell =0}^{j+3}
\max \{1, 2^{(\ell -j\rho)n/2}\}
2^{jm -k(N-n/2)} 
\|f_{\ell}\|_{L^2}
\|g\|_{L^{\infty}}
\|h_j \|_{L^2}
\nonumber 
\\
&\le 
\sum_{k\ge 0}\, 
\sum_{\substack{
j\ge 1, \, \ell \ge 0 
\\
\ell \le j+3
}}
\max \{1, 2^{(\ell -j\rho)n/2}\}
2^{jm -k(N-n/2)} 
\|f_{\ell}\|_{L^2}
\|g\|_{L^{\infty}}
\|h_j \|_{L^2}.
\end{align}
Since \eqref{Schur1} and \eqref{Schur2} hold 
if $\ell \le j+1$ is replaced by $\ell \le j+3$, 
by Schur's lemma, 
\eqref{case2-final} is bounded by 
\[
\sum_{k\ge 0}
2^{-k(N-n/2)} 
\bigg(
\sum_{\ell \ge 0} \|f_{\ell}\|_{L^2}^2
\bigg)^{1/2}
\bigg(
\sum_{j\ge 1} \|h_{j}\|_{L^2}^2
\bigg)^{1/2}
\|g\|_{L^{\infty}}
\lesssim 
\|f\|_{L^2} 
\|h\|_{L^2} 
\|g\|_{L^{\infty}}, 
\]
which gives the desired estimate for $Y_1$. 
This completes the proof of Theorem \ref{main-thm-1}.

\section{Proof of Theorem \ref{main-thm-2}}\label{section4}

In this section,
we shall prove Theorem \ref{main-thm-2}. 
The main scheme of the arguments is the same as 
that of Naibo \cite{Naibo}. 
In the last step, we introduce a new idea 
of using weak type estimates.

Since the theorem is already proved in the case $\rho=0$ 
(see \cite{Miyachi-Tomita}), 
for the rest of this section, we assume 
$0<\rho<1$, $m=-(1-\rho)n$, and 
$\sigma \in BS^{m}_{\rho, \rho}(\R^n)$.

Using the function $\Psi_j$ of Subsection \ref{subsection-decomposition}, 
we decompose $\sigma$ as 
\begin{align}
&
\sigma (x,\xi,\eta)
=
\sum_{j=0}^{\infty}
\sigma_{j} (x,\xi,\eta), 
\label{sigma-sum-sigmaj}
\\
&
\sigma_{j} (x,\xi,\eta)
=\sigma (x,\xi,\eta)
\Psi_{j}(\xi, \eta). 
\label{def-sigmaj}
\end{align}
We write the inverse Fourier transform of $\sigma_{j}$ 
with respect to $(\xi, \eta)$ as 
\begin{equation*}
K_{j} (x,y,z)=\frac{1}{(2\pi )^{2n}}
\int_{\R^n \times \R^n} 
e^{i(y \cdot \xi + z \cdot \eta)} 
\sigma_{j} (x,\xi,\eta)\, d\xi d\eta.  
\end{equation*}

First, we shall prove that $K_j$ satisfy the following estimates: 
\begin{align}
&
\|(1+2^{j\rho}|y|)^{N_1}
(1+2^{j\rho}|z|)^{N_2}
K_{j}(x,y,z)
\|_{L^2_{y,z}}
\lesssim 
2^{j (m+n)}, 
\label{weightKj-0}
\\
&
\|(1+2^{j\rho}|y|)^{N_1}
(1+2^{j\rho}|z|)^{N_2}
\nabla_{x} 
K_{j}(x,y,z)
\|_{L^2_{y,z}}
\lesssim 
2^{j (\rho + m+n)}, 
\label{weightKj-nabla-x}
\\
&
\|(1+2^{j\rho}|y|)^{N_1}
(1+2^{j\rho}|z|)^{N_2}
\nabla_{y} 
K_{j}(x,y,z)
\|_{L^2_{y,z}}
\lesssim 
2^{j (1+m+n)}, 
\label{weightKj-nabla-y}
\\
&
\|(1+2^{j\rho}|y|)^{N_1}
(1+2^{j\rho}|z|)^{N_2}
\nabla_{z} 
K_{j}(x,y,z)
\|_{L^2_{y,z}}
\lesssim 
2^{j (1+m+n)}, 
\label{weightKj-nabla-z}
\end{align}
where $\nabla_{x}, \nabla_{y}, \nabla_{z}$ denote the gradient 
operator with respect to $x,y,z$ respectively, 
and $N_1$ and $N_2$ can be arbitrary nonnegative real numbers.

To prove \eqref{weightKj-0}, observe that 
$1+|\xi|+|\eta| \approx 2^{j}$ for  
all $(\xi,\eta)\in \suppast (\sigma_j)$ 
and $\sigma_j$ satisfies the estimate  
\[
|\partial_{\xi}^{\beta}
\partial_{\eta}^{\gamma}
\sigma_{j} (x,\xi,\eta)|
\lesssim  
(2^j)^{m-\rho |\beta|- \rho |\gamma|}
\ichi \{1+|\xi|+|\eta|\approx 2^j\}. 
\]
Taking inverse Fourier transform with respect to $(\xi, \eta)$ 
and using Plancherel's theorem, 
we obtain 
\[
\|
(2^{j\rho}y)^{\beta}
(2^{j\rho}z)^{\gamma}
K_{j} (x,y,z)\|_{L^{2}_{y,z}}
\lesssim  
(2^j)^{m+ n},  
\]
from which \eqref{weightKj-0} follows. 
The estimates \eqref{weightKj-nabla-x}, 
\eqref{weightKj-nabla-y}, and \eqref{weightKj-nabla-z} 
can be derived from the estimates 
\begin{align*}
&
|\partial_{\xi}^{\beta}
\partial_{\eta}^{\gamma}
\nabla_{x}\sigma_{j} (x,\xi,\eta)|
\lesssim  
(2^j)^{m+\rho-\rho |\beta|- \rho |\gamma|}
\ichi \{1+|\xi|+|\eta|\approx 2^j\},  
\\
&
|\partial_{\xi}^{\beta}
\partial_{\eta}^{\gamma}
\{\xi \sigma_{j} (x,\xi,\eta)\}|
\lesssim  
(2^j)^{m+1 -\rho |\beta|- \rho |\gamma|}
\ichi \{1+|\xi|+|\eta|\approx 2^j\},  
\\
&
|\partial_{\xi}^{\beta}
\partial_{\eta}^{\gamma}
\{\eta \sigma_{j} (x,\xi,\eta)\}|
\lesssim  
(2^j)^{m+1 -\rho |\beta|- \rho |\gamma|}
\ichi \{1+|\xi|+|\eta|\approx 2^j\} 
\end{align*}
in the same way.

Now we proceed to the proof of the $L^{\infty} \times L^{\infty} \to BMO$ 
boundedness of $T_{\sigma}$. 
Let $f,g$ be functions satisfying 
$\|f\|_{L^{\infty}}=\|g\|_{L^{\infty}}=1$ 
and let $Q$ be a cube in $\R^n$. 
We denote by $\ell(Q)$ the side length of $Q$,
and by $x_Q$ the center of $Q$.
It is sufficient to prove that there exists a complex number 
$C_{Q}$ such that 
\begin{equation*}
\frac{1}{|Q|} 
\int_{Q} 
|T_{\sigma}(f,g)(x) - C_{Q}|\, dx \lesssim 1. 
\end{equation*}
We write $h=\ell (Q)$ and take 
the cube $\widetilde{Q}$ with the same center 
as $Q$ and with the sidelength 
\[
\ell (\widetilde{Q})
=
\left\{
\begin{array}{ll}
{2 h^{\rho}} & {\qquad\text{if $h\le 1$,}\quad }\\ 
{2h } & {\qquad\text{if $h> 1$.}\quad }
\end{array}
\right.
\]
We divide $f$ and $g$ as  
\begin{align*}
&
f= f \ichi_{\widetilde{Q}} + f \ichi_{\widetilde{Q}^{c}} = f^{(0)}+ f^{(1)}, 
\\
&
g=g \ichi_{\widetilde{Q}} + g \ichi_{\widetilde{Q}^{c}} = g^{(0)}+ g^{(1)}, 
\end{align*}
and divide $T_{\sigma}(f,g)$ into four parts 
\begin{align*}
T_{\sigma}(f,g)&=
T_{\sigma}(f^{(0)}, g^{(0)})
+
T_{\sigma}(f^{(0)}, g^{(1)})
+T_{\sigma}(f^{(1)}, g^{(0)})
+T_{\sigma}(f^{(1)}, g^{(1)})
\\
&=
F^{(1)}+F^{(2)}+F^{(3)}+F^{(4)}, 
\quad \text{say}.  
\end{align*}
For each $i=1,2,3,4$, we shall show that there exists 
a complex number $C^{(i)}_{Q}$ such that 
\begin{equation}\label{BMOFi}
\frac{1}{|Q|} 
\int_{Q} 
|F^{(i)}(x) - C^{(i)}_{Q}|\, dx \lesssim 1. 
\end{equation}
We divide the argument into two cases, 
$h> 1$ and $h\le 1$. 
\vspace{12pt}

\noindent {\bf The case $h=\ell (Q)>1$.}  

In this case, we shall prove 
\eqref{BMOFi} with $C^{(i)}_{Q}=0$ for all $i$.

{\it Estimate for $F^{(4)}$.\/} 
We have 
\begin{equation*}
F^{(4)}(x)
=T_{\sigma}(f^{(1)}, g^{(1)})(x)
=\sum_{j=0}^{\infty}T_{\sigma_j}(f^{(1)}, g^{(1)})(x). 
\end{equation*}
Using the kernel $K_j$ and using 
Schwarz's inequality, 
we have 
\begin{align*}
&
|T_{\sigma_{j}}(f^{(1)}, g^{(1)})(x)|
=\bigg|
\int_{\substack{
y\in \widetilde{Q}^{c}
\\
z\in \widetilde{Q}^{c}
}}
K_{j}(x,x-y, x-z)
f(y)g(z)\, dydz
\bigg| 
\\
&\le 
\bigg\|
h^{n} 
\bigg(\frac{|x-y|}{h}\bigg)^{N_{1}}
\bigg(\frac{|x-z|}{h}\bigg)^{N_{2}}
K_{j}(x,x-y, x-z)
\bigg\|_{L^2( y\in \widetilde{Q}^{c},\, z\in \widetilde{Q}^c)}
\\
&\quad 
\times \bigg\|
h^{-n} 
\bigg(\frac{|x-y|}{h}\bigg)^{-N_1}
\bigg(\frac{|x-z|}{h}\bigg)^{-N_2}
f(y) g(z) 
\bigg\|_{L^2( y\in \widetilde{Q}^{c},\, z\in \widetilde{Q}^c)},  
\end{align*}
where $N_1, N_2 \ge 0$ can be taken arbitrarily. 
The first $L^2$-norm above is estimated by \eqref{weightKj-0} as 
\begin{align*}
&\bigg\|
h^{n} 
\bigg(\frac{|x-y|}{h}\bigg)^{N_1}
\bigg(\frac{|x-z|}{h}\bigg)^{N_2}
K_{j}(x,x-y, x-z)
\bigg\|_{L^2( y\in \widetilde{Q}^{c},\, z\in \widetilde{Q}^{c})}
\\
&\le  
h^{n}
(2^{j\rho} h)^{-N_1}
(2^{j\rho} h)^{-N_2}
\|(2^{j\rho}|y|)^{N_1}(2^{j\rho}|z|)^{N_2}
K_{j}(x,y,z)
\|_{L^2_{y,z}}
\\
&\lesssim 
h^{n}
(2^{j\rho} h)^{-N_1}
(2^{j\rho} h)^{-N_2}
2^{j (m+n)}
=(2^{j\rho}h )^{-N_1 -N_2+n},   
\end{align*}
where the last equality holds because of our assumption $m=-(1-\rho)n$. 
If we take $N_1, N_2>n/2$, then, for $x\in Q$, 
the second $L^2$-norm is 
estimated as 
\begin{align*}
&
\bigg\|
h^{-n} 
\bigg(\frac{|x-z|}{h}\bigg)^{-N_1}
\bigg(\frac{|x-z|}{h}\bigg)^{-N_2}
f(y) g(z) 
\bigg\|_{L^2( y\in \widetilde{Q}^{c},\, z\in \widetilde{Q}^{c})}
\\
&\le 
\bigg\|
h^{-n} 
\bigg(\frac{|x-y|}{h}\bigg)^{-N_1}
\bigg(\frac{|x-z|}{h}\bigg)^{-N_2}
\bigg\|_{L^2( y\in \widetilde{Q}^{c},\, z\in \widetilde{Q}^{c})}
\approx 1. 
\end{align*}
Thus, by taking $N_1, N_2 >n/2$, we obtain 
the pointwise estimate 
\begin{align*}
&
|F^{(4)}(x)|
\le  
\sum_{j=0}^{\infty}
|T_{\sigma_j}(f^{(1)}, g^{(1)})(x)|
\lesssim 
\sum_{j=0}^{\infty}
(2^{j\rho}h )^{-N_1-N_2+n}
\approx h^{-N_1-N_2+n}\le 1
\end{align*}
for all $x\in Q$. 
This certainly implies \eqref{BMOFi} for $i=4$ with $C^{(4)}_{Q}=0$.

{\it Estimate for $F^{(2)}$ and $F^{(3)}$.\/} 
By symmetry, we consider only $F^{(2)}$. 
We write 
\begin{equation*}
F^{(2)}(x)
=T_{\sigma}(f^{(0)}, g^{(1)})(x)
=\sum_{j=0}^{\infty}T_{\sigma_j}(f^{(0)}, g^{(1)})(x). 
\end{equation*}
By Schwarz's inequality, 
we have 
\begin{align*}
&|T_{\sigma_{j}}(f^{(0)}, g^{(1)})(x)|=
\bigg|
\int_{\substack{
y\in \widetilde{Q}
\\
z\in \widetilde{Q}^{c}
}}
K_{j}(x,x-y, x-z)
f(y)g(z)\, dydz
\bigg| 
\\
&\le 
\bigg\|
h^{n} \bigg(\frac{|x-z|}{h}\bigg)^{N_{2}}
K_{j}(x,x-y, x-z)
\bigg\|_{L^2( y\in \widetilde{Q},\, z\in \widetilde{Q}^c)}
\\
&\quad 
\times \bigg\|
h^{-n} 
\bigg(\frac{|x-z|}{h}\bigg)^{-N_2}
f(y) g(z) 
\bigg\|_{L^2( y\in \widetilde{Q},\, z\in \widetilde{Q}^c)},  
\end{align*}
where $N_2 \ge 0$ can be taken arbitrarily. 
The first $L^2$-norm above is estimated by \eqref{weightKj-0} as 
\begin{align*}
&\bigg\|
h^{n} 
\bigg(\frac{|x-z|}{h}\bigg)^{N_2}
K_{j}(x,x-y, x-z)
\bigg\|_{L^2( y\in \widetilde{Q},\, z\in \widetilde{Q}^{c})}
\\
&\le 
h^{n}
(2^{j\rho} h)^{-N_2}
\|(2^{j\rho}|z|)^{N_2}
K_{j}(x,y,z)
\|_{L^2_{y,z}}
\\
&\lesssim 
h^{n}
(2^{j\rho} h)^{-N_2}
2^{j (m+n)}
=(2^{j\rho}h )^{-N_2+n}.  
\end{align*}
If we take $N_2>n/2$, then, for $x\in Q$, 
the second $L^2$-norm is 
estimated as 
\begin{align*}
&
\bigg\|
h^{-n} 
\bigg(\frac{|x-z|}{h}\bigg)^{-N_2}
f(y) g(z) 
\bigg\|_{L^2( y\in \widetilde{Q},\, z\in \widetilde{Q}^{c})}
\\
&\le 
\bigg\|
h^{-n} 
\bigg(\frac{|x-z|}{h}\bigg)^{-N_2}
\bigg\|_{L^2( y\in \widetilde{Q},\, z\in \widetilde{Q}^{c})}
\approx 1. 
\end{align*}
Thus, by taking $N_2 >n$, we obtain 
\begin{align*}
&
|F^{(2)}(x)|
\le  
\sum_{j=0}^{\infty}
|T_{\sigma_j}(f^{(0)}, g^{(1)})(x)|
\lesssim 
\sum_{j=0}^{\infty}
(2^{j\rho}h )^{-N_2+n}
\approx h^{-N_2+n}\le 1
\end{align*}
for all $x\in Q$. 
This implies \eqref{BMOFi} for $i=2$ with $C^{(2)}_{Q}=0$.

{\it Estimate for $F^{(1)}$.\/} 
Since $m=-(1-\rho)n<-(1-\rho)n/2=m_{\rho}(2,2)$, 
the operator $T_{\sigma}$ is bounded in $L^{2}\times L^{2} \to L^{1}$ 
(see Proposition \ref{critical-order}). 
Hence 
\[
\frac{1}{|Q|} 
\int_{Q} 
|F^{(1)}(x)|\, dx 
\le |Q|^{-1}
\|T_{\sigma}(f^{(0)}, g^{(0)})\|_{L^1}
\lesssim 
|Q|^{-1}\|f^{(0)}\|_{L^2} 
\|g^{(0)}\|_{L^2} 
\lesssim 
1.   
\]
\vspace{12pt}

\noindent {\bf The case $h=\ell (Q)\le 1$.}

{\it Estimate for $F^{(4)}$.\/} 
We shall prove the estimate \eqref{BMOFi} for $i=4$ with 
$C^{(4)}_{Q}=F^{(4)}(x_Q)$. 
In the following, $x$ always denotes arbitrary point in $Q$.

To estimate 
$F^{(4)}(x)-F^{(4)}(x_Q)$, 
we write 
\begin{align*}
&F^{(4)}(x)-F^{(4)}(x_Q)
= 
\sum_{j=0}^{\infty}
(T_{\sigma_j}(f^{(1)}, g^{(1)})(x)
-
T_{\sigma_j}(f^{(1)}, g^{(1)})(x_{Q}))
\\
&
=\sum_{j=0}^{\infty}
\int_{
\substack{
y\in \widetilde{Q}^{c}
\\
z\in \widetilde{Q}^{c}
}}
H_{j,Q}(x,y, z)
f(y)g(z)\, dydz, 
\end{align*}
where 
\begin{equation}\label{def-HjQ}
H_{j,Q}(x,y,z)
=K_{j}(x,x-y, x-z)- 
K_{j}(x_{Q},x_{Q}-y, x_{Q}-z).  
\end{equation}

By Schwarz's inequality, 
\begin{equation}\label{F4jx-F4jxQ-Schwarz}
\begin{split}
&
|T_{\sigma_j}(f^{(1)}, g^{(1)})(x) - T_{\sigma_j}(f^{(1)}, g^{(1)})(x_{Q})|
\\
&
\le 
\bigg\|
h^{\rho n}
\bigg(
\frac{|x-y|}{h^\rho }
\bigg)^{N_1}
\bigg(
\frac{|x-z|}{h^\rho }
\bigg)^{N_2}
H_{j,Q}(x,y,z)
\bigg\|_{L^2 (y\in \widetilde{Q}^{c},\,z\in \widetilde{Q}^{c})}
\\
&\quad 
\times 
\bigg\|
h^{-\rho n}
\bigg(
\frac{|x-y|}{h^\rho }
\bigg)^{-N_1}
\bigg(
\frac{|x-z|}{h^\rho }
\bigg)^{-N_2}
f(y) g(z) 
\bigg\|_{L^2 (y\in \widetilde{Q}^{c},\,z\in \widetilde{Q}^{c})}. 
\end{split}
\end{equation}

Since $\|f\|_{\infty}=\|g\|_{\infty}=1$, 
if we take $N_1, N_2 >n/2$, 
the latter $L^2$-norm of \eqref{F4jx-F4jxQ-Schwarz} 
is $\lesssim 1$.

In order to estimate the former 
$L^2$-norm of \eqref{F4jx-F4jxQ-Schwarz}, 
we write 
\begin{equation}\label{def-HjQ-int}
H_{j,Q}(x,y,z)
=\int_{0}^{1}
\nabla K_{j}(x(t),x(t)-y, x(t)-z)
\cdot (x-x_{Q}, x-x_{Q}, x-x_{Q})
\, dt,  
\end{equation} 
where we used the notation $x(t)=x_{Q}+t (x-x_Q)$ and 
\begin{align*}
&\nabla K_{j}(x,y, z)
\cdot (u, v, w)
\\
&=
\nabla_{x} K_{j}(x,y, z)
\cdot u
+ \nabla_{y} K_{j}(x,y, z)
\cdot v
+
\nabla_{z} K_{j}(x,y, z)
\cdot w. 
\end{align*}
Notice that $|x-y|\approx |x(t)-y|$ and 
$|x-z|\approx |x(t)-z|$ for all 
$y,z \in \widetilde{Q}^{c}$ and 
$0<t<1$. 
Hence, by \eqref{def-HjQ-int} and by 
\eqref{weightKj-nabla-x}, 
\eqref{weightKj-nabla-y}, \eqref{weightKj-nabla-z}, 
we can estimate the former 
$L^2$-norm of \eqref{F4jx-F4jxQ-Schwarz} 
as follows: 
(here $\|\cdots \|_{L^2(\ast)}$ means 
$\|\cdots \|_{L^2 (y\in \widetilde{Q}^{c},\,z\in \widetilde{Q}^{c})}$) 
\begin{align*}
&
\bigg\|
h^{\rho n}
\bigg(
\frac{|x-y|}{h^\rho }
\bigg)^{N_1}
\bigg(
\frac{|x-z|}{h^\rho }
\bigg)^{N_2}
H_{j,Q}(x,y,z)
\bigg\|_{L^2 (\ast)}
\\
&\lesssim 
\bigg\|
h^{1+\rho n}
\bigg(
\frac{|x-y|}{h^\rho }
\bigg)^{N_1}
\bigg(
\frac{|x-z|}{h^\rho }
\bigg)^{N_2}
\int_{0}^{1}
|\nabla K_{j}(x(t),x(t)-y, x(t)-z)|\, dt
\bigg\|_{L^2 (\ast)}
\\
&\approx 
\bigg\|
h^{1+\rho n} 
\int_{0}^{1}
\bigg(
\frac{|x(t)-y|}{h^\rho }
\bigg)^{N_1}
\bigg(
\frac{|x(t)-z|}{h^\rho }
\bigg)^{N_2}
|\nabla K_{j}(x(t),x(t)-y, x(t)-z)|
\, dt
\bigg\|_{L^2 (\ast)}
\\
&
\le 
h^{1+\rho n} \int_{0}^{1} 
\bigg\|
\bigg(
\frac{|x(t)-y|}{h^\rho }
\bigg)^{N_1}
\bigg(
\frac{|x(t)-z|}{h^\rho }
\bigg)^{N_2}
\nabla K_{j}(x(t), x(t)-y, x(t)-z)
\bigg\|_{L^2 (\ast)}
\, dt
\\
& 
\lesssim 
h^{1+\rho n} 
(2^{j \rho}h^{\rho})^{-N_1}
(2^{j \rho}h^{\rho})^{-N_2}
2^{j(1+m+n)}
=(2^{j \rho}h^{\rho})^{-N_1 -N_2 +n +1/\rho}, 
\end{align*}
where we used the assumption $m=-(1-\rho)n$ to obtain the last equality.  
On the other hand, if we use \eqref{weightKj-0}, then we can estimate 
the former $L^2$-norm of \eqref{F4jx-F4jxQ-Schwarz} as 
follows: (the notation $\|\cdots \|_{L^2(\ast)}$ is the same as above)
\begin{align*}
&\bigg\|
h^{\rho n}
\bigg(
\frac{|x-y|}{h^\rho }
\bigg)^{N_1}
\bigg(
\frac{|x-z|}{h^\rho }
\bigg)^{N_2}
H_{j,Q}(x,y,z)
\bigg\|_{L^2 (\ast)}
\\
&
\le 
\bigg\|
h^{\rho n}
\bigg(
\frac{|x-y|}{h^\rho }
\bigg)^{N_1}
\bigg(
\frac{|x-z|}{h^\rho }
\bigg)^{N_2}
K_{j}(x,x-y, x-z)
\bigg\|_{L^2 (\ast)}
\\
&\quad 
+ \bigg\|
h^{\rho n}
\bigg(
\frac{|x-y|}{h^\rho }
\bigg)^{N_1}
\bigg(
\frac{|x-z|}{h^\rho }
\bigg)^{N_2}
K_{j}(x_{Q},x_{Q}-y, x_{Q}-z)
\bigg\|_{L^2 (\ast)}
\\
&\approx 
\bigg\|
h^{\rho n}
\bigg(
\frac{|x-y|}{h^\rho }
\bigg)^{N_1}
\bigg(
\frac{|x-z|}{h^\rho }
\bigg)^{N_2}
K_{j}(x,x-y, x-z)
\bigg\|_{L^2 (\ast)}
\\
&\quad 
+
\bigg\|
h^{\rho n}
\bigg(
\frac{|x_{Q}-y|}{h^\rho }
\bigg)^{N_1}
\bigg(
\frac{|x_{Q}-z|}{h^\rho }
\bigg)^{N_2}
K_{j}(x_{Q},x_{Q}-y, x_{Q}-z)
\bigg\|_{L^2 (\ast)}
\\
&
\lesssim 
h^{\rho n} 
(2^{j \rho}h^{\rho})^{-N_1}
(2^{j \rho}h^{\rho})^{-N_2}
2^{j(m+n)}
=(2^{j \rho}h^{\rho})^{-N_1 -N_2 +n}.  
\end{align*}

Combining the above estimates, 
we have the following estimates for arbitrary 
$N_1, N_2  >n/2$:  
\[
|T_{\sigma_j}(f^{(1)}, g^{(1)})(x) - T_{\sigma_j}(f^{(1)}, g^{(1)})(x_{Q})|
\lesssim 
\min \{(2^{j \rho}h^{\rho})^{-N_1 -N_2 +n +1/\rho},\,
(2^{j \rho}h^{\rho})^{-N_1 -N_2 +n}\}. 
\]

Now we take $N_{1}=N_{2}=N$ such that 
$-2N + n +1/\rho >0 > -2N +n$. 
Then taking the sum of 
the above estimates over $j\ge 0$, we obtain 
\[
|F^{(4)}(x)- F^{(4)}(x_Q)|
\le 
\sum_{j=0}^{\infty}
|T_{\sigma_j}(f^{(1)}, g^{(1)})(x) - T_{\sigma_j}(f^{(1)}, g^{(1)})(x_{Q})|
\lesssim 1, 
\quad x\in Q, 
\]
which a fortiori implies 
\eqref{BMOFi} for $i=4$ with $C^{(4)}_{Q}=F^{(4)}(x_Q)$.

{\it Estimate for $F^{(2)}$ and $F^{(3)}$.\/} 
By symmetry, we consider only $F^{(2)}$. 
We shall prove the estimate \eqref{BMOFi} for $i=2$ with 
$C^{(2)}_{Q}=F^{(2)}(x_Q)$. 
In the following, $x$ always denotes arbitrary point in $Q$.

We write 
\begin{align*}
&F^{(2)}(x)-F^{(2)}(x_Q)
= 
\sum_{j=0}^{\infty}
(T_{\sigma_j}(f^{(0)}, g^{(1)})(x)
-
T_{\sigma_j}(f^{(0)}, g^{(1)})(x_{Q}))
\\
&
=
\sum_{j=0}^{\infty}
\int_{
\substack{
y\in \widetilde{Q}
\\
z\in \widetilde{Q}^{c}
}}
H_{j,Q}(x,y, z)
f(y)g(z)\, dydz  
\end{align*}
with $H_{j,Q}$ given by \eqref{def-HjQ}.

By Schwarz's inequality, 
we have 
\begin{equation}\label{F2jx-F2jxQ-Schwarz}
\begin{split}
&
|T_{\sigma_j}(f^{(0)}, g^{(1)})(x) - T_{\sigma_j}(f^{(0)}, g^{(1)})(x_{Q})|
\\
&
\le 
\bigg\|
h^{\rho n}
\bigg(
\frac{|x-z|}{h^\rho }
\bigg)^{N_2}
H_{j,Q}(x,y,z)
\bigg\|_{L^2 (y\in \widetilde{Q},\,z\in \widetilde{Q}^{c})}
\\
&\quad 
\times 
\bigg\|
h^{-\rho n}
\bigg(
\frac{|x-z|}{h^\rho }
\bigg)^{-N_2}
f(y) g(z) 
\bigg\|_{L^2 (y\in \widetilde{Q},\,z\in \widetilde{Q}^{c})}. 
\end{split}
\end{equation}

Since $\|f\|_{\infty}=\|g\|_{\infty}=1$, 
if we take $N_2 >n/2$, then 
the latter $L^2$-norm of \eqref{F2jx-F2jxQ-Schwarz} 
is $\lesssim 1$.

By \eqref{def-HjQ-int} and by 
\eqref{weightKj-nabla-x}, \eqref{weightKj-nabla-y}, and 
\eqref{weightKj-nabla-z}, 
we can estimate the former 
$L^2$-norm of \eqref{F2jx-F2jxQ-Schwarz} 
as  
\begin{align*}
&
\bigg\|
h^{\rho n}
\bigg(
\frac{|x-z|}{h^\rho }
\bigg)^{N_2}
H_{j,Q}(x,y,z)
\bigg\|_{L^2 (y\in \widetilde{Q},\,z\in \widetilde{Q}^{c})}
\\
&\lesssim 
\bigg\|
h^{1+\rho n}
\bigg(
\frac{|x-z|}{h^\rho }
\bigg)^{N_2}
\int_{0}^{1}
|\nabla K_{j}(x(t),x(t)-y, x(t)-z)|\, dt
\bigg\|_{L^2 (y\in \widetilde{Q},\,z\in \widetilde{Q}^{c})}
\\
&\approx 
\bigg\|
h^{1+\rho n} 
\int_{0}^{1}
\bigg(
\frac{|x(t)-z|}{h^\rho }
\bigg)^{N_2}
|\nabla K_{j}(x(t),x(t)-y, x(t)-z)|
\, dt
\bigg\|_{L^2 (y\in \widetilde{Q},\,z\in \widetilde{Q}^{c})}
\\
&
\le 
h^{1+\rho n} \int_{0}^{1} 
\bigg\|
\bigg(
\frac{|x(t)-z|}{h^\rho }
\bigg)^{N_2}
\nabla K_{j}(x(t), x(t)-y, x(t)-z)
\bigg\|_{L^2 (y\in \widetilde{Q},\,z\in \widetilde{Q}^{c})}
\, dt
\\
& 
\lesssim 
h^{1+\rho n} 
(2^{j \rho}h^{\rho})^{-N_2}
2^{j(1+m+n)}
=(2^{j \rho}h^{\rho})^{ -N_2 +n +1/\rho}. 
\end{align*}
On the other hand, using \eqref{weightKj-0}, we can estimate 
the former $L^2$-norm of \eqref{F2jx-F2jxQ-Schwarz} as 
\begin{align*}
&\bigg\|
h^{\rho n}
\bigg(
\frac{|x-z|}{h^\rho }
\bigg)^{N_2}
H_{j,Q}(x,y,z)
\bigg\|_{L^2 (y\in \widetilde{Q},\,z\in \widetilde{Q}^{c})}
\\
&
\le 
\bigg\|
h^{\rho n}
\bigg(
\frac{|x-z|}{h^\rho }
\bigg)^{N_2}
K_{j}(x,x-y, x-z)
\bigg\|_{L^2 (y\in \widetilde{Q},\,z\in \widetilde{Q}^{c})}
\\
&\quad 
+ \bigg\|
h^{\rho n}
\bigg(
\frac{|x-z|}{h^\rho }
\bigg)^{N_2}
K_{j}(x_{Q},x_{Q}-y, x_{Q}-z)
\bigg\|_{L^2 (y\in \widetilde{Q},\,z\in \widetilde{Q}^{c})}
\\
&\approx 
\bigg\|
h^{\rho n}
\bigg(
\frac{|x-z|}{h^\rho }
\bigg)^{N_2}
K_{j}(x,x-y, x-z)
\bigg\|_{L^2 (y\in \widetilde{Q},\,z\in \widetilde{Q}^{c})}
\\
&\quad 
+
\bigg\|
h^{\rho n}
\bigg(
\frac{|x_{Q}-z|}{h^\rho }
\bigg)^{N_2}
K_{j}(x_{Q},x_{Q}-y, x_{Q}-z)
\bigg\|_{L^2 (y\in \widetilde{Q},\,z\in \widetilde{Q}^{c})}
\\
&
\lesssim 
h^{\rho n} 
(2^{j \rho}h^{\rho})^{-N_2}
2^{j(m+n)}
=(2^{j \rho}h^{\rho})^{-N_2 +n}.  
\end{align*}

Combining the above estimates, 
we have the estimates 
\[
|T_{\sigma_j}(f^{(0)}, g^{(1)})(x) - T_{\sigma_j}(f^{(0)}, g^{(1)})(x_{Q})|
\lesssim 
\min \{(2^{j \rho}h^{\rho})^{-N_2 +n +1/\rho},\,
(2^{j \rho}h^{\rho})^{-N_2 +n}\}  
\]
for arbitrary $N_2  >n/2$.

Now we take $N_{2}$ such that 
$-N_{2} + n +1/\rho >0 > -N_{2} +n$ and take the sum of 
the above estimates over $j\ge 0$ to obtain 
\[
|F^{(2)}(x)- F^{(2)}(x_Q)|
\le 
\sum_{j=0}^{\infty}
|T_{\sigma_j}(f^{(0)}, g^{(1)})(x) - T_{\sigma_j}(f^{(0)}, g^{(1)})(x_{Q})|
\lesssim 1, 
\quad x\in Q, 
\]
which a fortiori implies 
\eqref{BMOFi} for $i=2$ with $C^{(2)}_{Q}=F^{(2)}(x_Q)$.

{\it Estimate for $F^{(1)}$.\/} 
We first prove an $L^2$ estimate of $T_{\sigma_j}(f^{(0)}, g^{(0)})$. 
Let $\widetilde{\sigma}_{j}$ be the symbol 
\begin{equation}\label{change-symbol}
\widetilde{\sigma}_{j}(x,\xi,\eta)
=
\sigma_{j}(2^{-j\rho}x,2^{j\rho}\xi,2^{j\rho}\eta). 
\end{equation}
Then a simple change of variables gives 
\begin{equation}\label{change-pdo}
T_{\sigma}(a,b)(2^{-j\rho} x)
=
T_{\widetilde{\sigma}}(a(2^{-j\rho}\cdot ),b(2^{-j\rho}\cdot ))
(x), 
\end{equation}
which implies 
\[
\|T_{\sigma_{j}}\|_{L^2 \times L^{\infty} \to L^2}
=
\|T_{\widetilde{\sigma}_{j}}\|_{L^2 \times L^{\infty} \to L^2}. 
\]
Since $1+|\xi|+|\eta|\approx 2^{j}$ for all 
$(\xi, \eta)\in \suppast (\sigma_{j})$, 
we see that $\widetilde{\sigma}_{j}$ satisfies 
\begin{equation}\label{change-est}
\begin{split}
|\partial_{x}^{\alpha}
\partial_{\xi}^{\beta}
\partial_{\eta}^{\gamma}
\widetilde{\sigma}_{j} (x,\xi, \eta)|
&\lesssim 
2^{jm}\ichi \{1+|\xi|+|\eta|\approx 2^{j(1-\rho)}\}
\\
&\lesssim  
2^{-j(1-\rho)n/2}
(1+|\xi|+|\eta|)^{-n/2}. 
\end{split}
\end{equation}
Hence the theorem of \cite{Miyachi-Tomita} or the case 
$\rho =0$ of Theorem \ref{main-thm-1} yields 
\[
\|T_{\widetilde{\sigma}_{j}}\|_{L^2 \times L^{\infty} \to L^2}\lesssim 
2^{-j(1-\rho)n/2}.
\] 
Thus we obtain 
\begin{equation}\label{L2-estimate}
\begin{split}
&
\|T_{\sigma_{j}} (f^{(0)}, g^{(0)})\|_{L^2}
\lesssim 
2^{-j(1-\rho)n/2}
\|f^{(0)}\|_{L^2}\|g^{(0)}\|_{L^{\infty}}
\\
&\le 
2^{-j(1-\rho)n/2}
|\widetilde{Q}|^{1/2}
\approx 
2^{-j(1-\rho)n/2}
|Q|^{\rho /2}. 
\end{split}
\end{equation}

Next we prove an $L^{\infty}$ estimate of $T_{\sigma_j}(f^{(0)}, g^{(0)})$. 
From the formula 
\[
T_{\sigma_j}(a,b)(x)
=\int_{\R^n} 
K_{j}(x,x-y,x-z) a(y) b(z)\, dydz
\]
and from \eqref{weightKj-0}, we have 
\begin{align*}
&|T_{\sigma_j}(a,b)(x)|
\le 
\|K_{j}(x,x-y,x-z) \|_{L^2_{y,z}}
\|a(y)b(z)\|_{L^2_{y,z}}
\\
&\lesssim 
2^{j (m+n)} 
\|a\|_{L^2}
\|b\|_{L^2}
=2^{j\rho n}
\|a\|_{L^2}
\|b\|_{L^2}. 
\end{align*}
Hence  
\begin{equation}\label{Linfty-estimate}
\begin{split}
&
\|T_{\sigma_{j}} (f^{(0)}, g^{(0)})\|_{L^{\infty}}
\lesssim 
2^{j\rho n}
\|f^{(0)}\|_{2}\|g^{(0)}\|_{2}
\\
&\le 
2^{j \rho n}
|\widetilde{Q}|
\approx 
2^{j \rho n}
|Q|^{\rho}. 
\end{split}
\end{equation}

Now by a characterization of weak $L^p$ functions 
(see Lemma \ref{weakLp} to be given below), 
the estimates \eqref{L2-estimate} and \eqref{Linfty-estimate} 
imply the following weak type estimate 
for $F^{(1)}=\sum_{j=0}^{\infty}T_{\sigma_{j}} (f^{(0)}, g^{(0)})$: 
\[
|\{x \in \R^{n}
\,:\, 
|F^{(1)}(x)| > \lambda 
\}|
\lesssim 
|Q| \lambda^{-1-1/\rho}, \quad \lambda >0.  
\]
From this we obtain  
\begin{align*}
\frac{1}{|Q|}\int_{Q} 
|F^{(1)}(x)|\, 
dx
&=
\int_{0}^{\infty}
|Q|^{-1}
|\{x \in Q
\,:\, 
|F^{(1)}(x)| > \lambda 
\}|\, d\lambda 
\\
&\le 
\int_{0}^{\infty}
\min \{
1, 
\lambda^{-1-1/\rho}\}\, d\lambda 
\approx 1,  
\end{align*}
which is the estimate \eqref{BMOFi} for $i=1$ and $C^{(1)}_Q=0$. 
This completes the proof of Theorem \ref{main-thm-2}. 
\vspace{12pt}

Finally we shall give a proof of the fact that was used 
at the last part of the above argument. 
Here we shall give a slightly general lemma. 
This lemma is equivalent to the fact that the space  
$L^{(p,\infty)}$ is equal to the real interpolation space 
$[L^{\infty}, L^{r}]_{\theta, \infty}$, $1/p=\theta/r$, 
combined with the characterization of the latter space 
by the $J$-method. 
Although this may be known to many people, we shall give a proof 
for reader's convenience.

\begin{lem}\label{weakLp}
Let $0<r<p<\infty$, $\alpha, \beta \in (0,\infty)$, and $0<\theta<1$ 
satisfy $1/p=\theta/r$ and $\alpha/ (\alpha+ \beta)=\theta$. 
Then for nonnegative measurable functions $f$ on a measure space 
the following two conditions are equivalent: 
\begin{itemize}
\item[(i)]  
there exists constants $A,B\in (0,\infty)$ and 
a sequence of nonnegative measurable functions $\{f_j\}_{j\in \Z}$ 
such that 
$\|f_j\|_{L^{\infty}} \le A 2^{j\alpha}$, 
$\|f_j\|_{L^{r}} \le B 2^{-j\beta}$, and $f=\sum_{j\in \Z} f_j$.  
\item[{\rm (ii)}] 
$f\in L^{(p,\infty)}$, i.e., there exists a constant $C\in (0,\infty)$ 
such that 
$|\{x \, :\, f(x)>\lambda\}| \le (C\lambda^{-1})^{p}$ 
for all $\lambda >0$. 
\end{itemize}
To be precise, if {\rm (i)} holds  
then {\rm (ii)} holds with $C=c(p,r,\alpha, \beta)A^{1-\theta}B^{\theta}$, 
and, conversely, 
if {\rm (ii)} holds then {\rm (i)} holds with $A, B\in (0, \infty)$ 
such that $A^{1-\theta}B^{\theta}=c(p,r,\alpha, \beta) C$. 
\end{lem}

\begin{proof}
${\rm (i)} \Rightarrow{\rm (ii)}$. 
Suppose {\rm (i)} holds and 
write $\gamma =\alpha + \beta$. 
Take an integer $j_0$ such that $A2^{j_0 \alpha}\approx B2^{-j_0 \beta}$ 
and set $C=A2^{j_0 \alpha}$. 
Then $C\approx A^{1-\theta} B^{\theta}$, 
$\|f_{j+j_0}\|_{L^{\infty}} \lesssim C 2^{j\gamma \theta}$, and 
$\|f_{j+j_0}\|_{L^{r}} \lesssim C 2^{-j\gamma (1-\theta)}$. 
For $\lambda \in (0,\infty)$ given, take an integer $j_1$ such that 
$C 2^{j_{1}\gamma \theta}\approx \lambda$ and decompose $f$ as 
\[
f=
\sum_{j\le j_1} f_{j+j_0} 
+
\sum_{j> j_1} f_{j+j_0} 
=f^{(0)}+ f^{(1)}. 
\]
Then 
$\|f^{(0)}\|_{L^{\infty}}\lesssim C 2^{j_{1}\gamma \theta} \approx \lambda$ 
and 
$\|f^{(1)}\|_{L^{r}}\lesssim C 2^{-j_{1}\gamma (1-\theta)} \approx 
C^{1/\theta} \lambda^{1-1/\theta}$. 
Hence, if we take a sufficiently large constant $c_{0}$, which depends only 
on $p,r,\alpha,\beta$, then we have  
\begin{align*}
&
|\{x \, :\, f(x)>c_{0}\lambda\}| 
\le
|\{x \, :\, f^{(1)}(x)>\lambda\}| 
\\
&\le 
\|f^{(1)}\|_{L^{r}}^{r}
\lambda^{-r}
\lesssim 
(C^{1/\theta} \lambda^{1-1/\theta})^{r} \lambda^{-r}
=(C \lambda^{-1})^{p}.
\end{align*}

${\rm (ii)} \Rightarrow {\rm (i)}$. 
Suppose {\rm (ii)} holds. 
Take an $A\in (0, \infty)$ and decompose $f$ as 
\[
f(x) = \sum_{j\in \Z} f_j (x), 
\quad 
f_j (x) = f(x) \ichi \{A2^{(j-1)\alpha} < f(x) \le A2^{j\alpha} \}. 
\]
Then $\|f_j\|_{L^{\infty}}\le A2^{j\alpha}$ 
and 
\[
\|f_j\|_{L^{r}}
\le 
A2^{j\alpha}
|\{x \, :\, f(x) > A2^{(j-1)\alpha} \}|^{1/r}
\le 
A2^{j\alpha}
(C  A^{-1}2^{-(j-1)\alpha} )^{p/r}
= B 2^{-j \beta} 
\] 
with 
$B\approx C^{1/\theta} A^{1-1/\theta} $.

The relations between the constants $A,B$, and $C$ are obvious from 
the above arguments. 
\end{proof}

\section{Proof of Corollary \ref{main-cor}}\label{section5}

It is known that there exist bijective mappings 
$\sigma \mapsto \sigma^{\ast 1}$ and 
$\sigma \mapsto \sigma^{\ast 2}$ 
of $BS^m_{\rho,\rho}$, $0 \le \rho <1$,
onto itself  such that 
\begin{equation}\label{eq310}
\int T_{\sigma}(f,g)(x)h(x)dx
=\int T_{\sigma^{\ast 1}}(h,g)(x)f(x)dx
=\int T_{\sigma^{\ast 2}}(f,h)(x)g(x)dx
\end{equation}
for all $f,g,h \in \calS$ (see \cite[Theorem 2.1]{BMNT}). 
By duality, 
\begin{equation*} 
\|T_{\sigma}\|_{L^2\times L^{\infty} \to L^2}
= 
\|T_{\sigma^{\ast 2}}\|_{L^2 \times L^2 \to L^1}
=
\|T_{\left(\sigma^{\ast 2}\right)^{\ast1}}\|_{L^{\infty} \times L^2 \to L^2}. 
\end{equation*}
In particular, if one of the above is finite, then 
the other two are also finite. 
Thus the desired result for $(p,q)=(2,\infty), (2,2), (\infty,2)$
follows from Theorem \ref{main-thm-1}.
Similarly, by the duality between $H^1$ and $BMO$,
\begin{equation*} 
\|T_{\sigma}\|_{L^{\infty}\times L^{\infty} \to BMO}
\approx  
\|T_{\sigma^{\ast 1}}\|_{H^1 \times L^{\infty} \to L^1}
\approx 
\|T_{\sigma^{\ast 2}}\|_{L^{\infty} \times H^1 \to L^1}. 
\end{equation*}
Hence the desired result for
$(p,q)=(\infty,\infty), (1,\infty), (\infty,1)$
follows from Theorem \ref{main-thm-2}.
Other cases can be obtained from interpolation.  
As for the interpolation argument, see for example 
\cite[Proof of Theorem 2.2]{BBMNT}.

\appendix\section{}\label{appendix}

In this appendix,
we shall prove Proposition \ref{critical-order}. 
Let $0<p,q,r \le \infty$ and $1/p+1/q=1/r$. 
We write 
$m_{0}=m_{0}(p,q)$. 
Recall that $m_{\rho}(p,q)=(1-\rho)m_0$. 
For simplicity of notation, we only consider the case $r<\infty$,
but the argument below works in the case $r=\infty$ as well. 
In fact, in the case $r=\infty$, all we need is to rewrite 
$L^{r}$ by $BMO$.

In \cite[Theorem A.2]{Miyachi-Tomita}, 
it is already proved that  
if $T_{\sigma}: H^p \times H^q \to L^r$ 
for all $m \in BS^{m}_{\rho,\rho}$ then 
$m \le (1-\rho)m_0$. 
Hence, in order to complete the proof of Proposition \ref{critical-order}, 
it is sufficient to show that 
if $m<(1-\rho)m_{0}$ then $T_{\sigma}: H^p \times H^q \to L^r $ 
for all $\sigma \in BS^m_{\rho,\rho}$. 
As we mentioned in Introduction, this has been proved 
in \cite{MRS} and \cite{BBMNT} 
in the range $1/p+1/q \le 1$. 
Here we shall give a proof that is valid for all $0<p,q \le \infty$.

We use the fact that the case $\rho=0$ is already known. 
To be precise, it is known that 
$T_{\sigma}: H^p \times H^q \to L^r $ 
for all $\sigma \in BS^{m_0}_{0,0}$ (see \cite[Theorem 1.1]{Miyachi-Tomita}).
By virtue of the closed graph theorem, 
this boundedness is equivalent to the claim that there exists
a positive integer $N$ and a constant $c$ such that
\begin{equation}\label{equiv-est}
\|T_{\sigma}\|_{H^p \times H^q \to L^r}
\le c \max_{|\alpha|, |\beta|, |\gamma| \le N}
\left(\sup_{x,\xi,\eta \in \R^n}
(1+|\xi|+|\eta|)^{-m_0}
|\partial^{\alpha}_x\partial^{\beta}_{\xi}
\partial^{\gamma}_{\eta}\sigma(x,\xi,\eta)|\right)
\end{equation}
for all $\sigma \in BS^{m_0}_{0,0}$
(see \cite[Lemma 2.6]{BBMNT}).

Now assume that $0<\rho<1$ and 
$\sigma \in BS^m_{\rho,\rho}$ with $m<(1-\rho)m_0$.
In the same way as in Section \ref{section4},
we write $\sigma = \sum_{j=0}^{\infty}\sigma_j$ 
as in \eqref{sigma-sum-sigmaj} and \eqref{def-sigmaj},  
and define $\widetilde{\sigma}_j$ by \eqref{change-symbol}. 
Then \eqref{change-pdo} holds and this, 
together with the relation $1/p + 1/q =1/r$, implies 
\begin{equation}\label{s-tildes}
\|T_{\sigma_j}\|_{H^p \times H^q \to L^r}
=
\|T_{\widetilde{\sigma}_j}\|_{H^p \times H^q \to L^r}. 
\end{equation}
Also, from the same argument as in \eqref{change-est}, 
we see that $\widetilde{\sigma}_j$ satisfies the 
estimate 
\[
|\partial^{\alpha}_x\partial^{\beta}_{\xi}
\partial^{\gamma}_{\eta}\widetilde{\sigma}_j(x,\xi,\eta)|
\le C_{\alpha,\beta,\gamma}\, 
2^{j(m-(1-\rho)m_{0})}
(1+|\xi|+|\eta|)^{m_0}.
\]
Combining this with \eqref{s-tildes} and \eqref{equiv-est},
we have 
\[
\|T_{{\sigma}_j}\|_{H^p \times H^q \to L^r}
=
\|T_{\widetilde{\sigma}_j}\|_{H^p \times H^q \to L^r}
\lesssim
2^{j (m- (1-\rho)m_0)}.  
\]
Since $m<(1-\rho)m_0$, the above inequality implies 
that $T_{\sigma}=\sum_{j=0}^{\infty}T_{\sigma_j}$ 
is bounded from $H^p \times H^q \to L^r$. 
This completes the proof of Proposition \ref{critical-order}.



\begin{thebibliography}{20}
\bibitem{BBMNT}
\'A. B\'enyi, F. Bernicot, D. Maldonado, V. Naibo and R. Torres,
{On the H\"ormander classes of bilinear pseudodifferential operators II},
Indiana Univ. Math. J. 62 (2013), 1733--1764.

\bibitem{BMNT}
\'A. B\'enyi, D. Maldonado, V. Naibo and R. Torres,
{On the H\"ormander classes of bilinear pseudodifferential operators},
Integral Equations Operator Theory 67 (2010), 341--364.

\bibitem{BT}
\'A. B\'enyi and R. Torres,
{Symbolic calculus and the transpose of
bilinear pseudodifferential operators},
Comm. PDE 28 (2003), 1161-1181.

\bibitem{BT-2}
\'A. B\'enyi and R. Torres,
{Almost orthogonality and a class of bounded bilinear
pseudodifferential operators},
Math. Res. Lett. 11 (2004), 1--12.

\bibitem{CV}
A.P. Calder\'on and R. Vaillancourt,
{A class of bounded pseudo-differential operators},
Proc. Nat. Acad. Sci. U.S.A. 69 (1972), 1185--1187.

\bibitem{CM}
R. Coifman and Y. Meyer,
{Au del\`a des op\'erateurs pseudo-diff\'erentiels},
Ast\'erisque 57 (1978), 1--185.

\bibitem{Grafakos-1}
L. Grafakos,
{Classical Fourier Analysis},
Second edition, Springer, New York, 2008.

\bibitem{GT}
L. Grafakos and R. Torres,
{Multilinear Calder\'on-Zygmund theory},
Adv. Math. 165 (2002), 124--164.

\bibitem{MRS}
N. Michalowski, D. Rule and W. Staubach,
{Multilinear pseudodifferential operators beyond
Calder\'on-Zygmund operators},
J. Math. Anal. Appl. 414 (2014), 149--165

\bibitem{Miyachi-Tomita}
A. Miyachi and N. Tomita,
{Calder\'on-Vaillancourt type theorem for bilinear operators},
Indiana Univ. Math. J. 62 (2013), 1165--1201.

\bibitem{Naibo}
V. Naibo,   
{On the $L^{\infty}\times L^{\infty} \to BMO$ 
mapping property for certain bilinear pseudodifferential operators}, 
Proc. Amer. Math. Soc. 143 (2015), 5323--5336. 

\bibitem{Stein}
E.M. Stein,
{Harmonic Analysis, Real Variable Methods,
Orthogonality, and Oscillatory Integrals},
Princeton University Press, Princeton, NJ, 1993.
\end{thebibliography}
\end{document}